\documentclass[intlim,righttag,10pt]{amsart}
\usepackage{amscd}
\usepackage{amssymb}
\usepackage[all]{xy}
\def\invlim{\mathop{\vtop{\ialign{##\crcr$\hfill{\lim}\hfil$\crcr
\noalign{\kern1pt\nointerlineskip}\leftarrowfill\crcr\noalign
{\kern -3pt}}}}\limits}
\def\dirlim{\mathop{\vtop{\ialign{##\crcr$\hfill{\lim}\hfil$\crcr
\noalign{\kern1pt\nointerlineskip}\rightarrowfill\crcr\noalign
{\kern -3pt}}}}\limits}
\def\lomapr#1{\smash{\mathop{\relbar\joinrel\longrightarrow}\limits^{#1}}}
 \def\verylomapr#1{\smash{\mathop{\relbar\joinrel\relbar\joinrel\relbar\joinrel\longrightarrow}\limits^{#1}}}
\def\veryverylomapr#1{\smash{\mathop{\relbar\joinrel\relbar\joinrel\relbar
\joinrel\relbar\joinrel\relbar\joinrel\longrightarrow}\limits^{#1}}}
\def\phi{\varphi}
\def\epsilon{\varepsilon}

\newtheorem{theorem}{Theorem}[section]
 \newtheorem{lemma}[theorem]{Lemma}
 \newtheorem{proposition}[theorem]{Proposition}
 \newtheorem{corollary}[theorem]{Corollary}

\theoremstyle{definition}
\newtheorem{definition}[theorem]{Definition}
\newtheorem{remark}[theorem]{Remark}

\newtheorem{example}[theorem]{Example}
\newtheorem*{acknowledgments}{Acknowledgments}
 \newcommand{\Nis}{\operatorname{Nis}}

\newcommand{\ovk}{\overline{K} }

\newcommand{\rig}{\operatorname{rig} } 
\newcommand{\spb}{\operatorname{sp} }

\newcommand{\gp}{\operatorname{gp} }

 \newcommand{\cl}{\operatorname{cl} } 
 \newcommand{\dr}{\operatorname{dR} }

 \newcommand{\holim}{\operatorname{holim} }

 \newcommand{\syneet}{\operatorname{s\acute{e}t} }
\newcommand{\synee}{\operatorname{s\acute{e}} }
 \newcommand{\eet}{\operatorname{\acute{e}t} }

 \newcommand{\dlog}{\operatorname{dlog} }

 \newcommand{\Zar}{\operatorname{Zar} }

 \newcommand{\conv}{\operatorname{conv} }
 \newcommand{\Spec}{\operatorname{Spec} }
 \newcommand{\fil}{\operatorname{Fil} }
 
 \newcommand{\Spf}{\operatorname{Spf} }
  
 \newcommand{\Hom}{\operatorname{Hom} }

 \newcommand{\tr}{ \operatorname{tr} }

 \newcommand{\id}{ \operatorname{Id} }
\newcommand{\synt}{ \operatorname{syn} }
 \newcommand{\res}{ \operatorname{res} }
 \newcommand{\Cone}{\operatorname{Cone} }
 
  \newcommand{\hk}{\operatorname{HK} }

 \newcommand{\kker}{\operatorname{Ker} }
 \newcommand{\crr}{\operatorname{cr} }

 \newcommand{\ve}{ \varepsilon  }
  \newcommand{\kr}{^{\scriptscriptstyle\bullet}}

 \newcommand{\sff}{{\mathcal{F}}}  
 
 \newcommand{\sy}{{\mathcal{Y}}}
 \newcommand{\sh}{{\mathcal{H}}}
 \newcommand{\sg}{{\mathcal{G}}}

 \newcommand{\scc}{{\mathcal{C}}}

 \newcommand{\so}{{\mathcal O}}
 \newcommand{\sj}{{\mathcal J}}
 \newcommand{\se}{{\mathcal{E}}}
 \newcommand{\sa}{{\mathcal{A}}}
 \newcommand{\szz}{{\mathcal{Z}}}
 \newcommand{\sx}{{\mathcal{X}}}
 \newcommand{\sss}{{\mathcal{S}}}

 \newcommand{\wh}{\widehat}

\newcommand{\Q}{\mathbf{Q}}
\newcommand{\Z}{\mathbf{Z}}
\newcommand{\N}{\mathbf{N}}
\newcommand{\R}{{\mathrm R}}
\newcommand{\M}{{\mathrm M}}
\newcommand{\F}{{\mathrm F}}

\numberwithin{equation}{section}
 
\begin{document}
 \title[Syntomic cohomology and $p$-adic motivic cohomology]
 {Syntomic cohomology and $p$-adic motivic cohomology}
 \author{Veronika Ertl}
 \address{Universit\"at Regensburg, Fakult\"at f\"ur Mathematik
Universit\"atsstrasse 31,
93053 Regensburg, Germany}
\email{veronika.ertl@mathematik.uni-regensburg.de}  
\author{Wies{\l}awa Nizio{\l}}
 \address{CNRS, UMPA, \'Ecole Normale Sup\'erieure de Lyon, 46 all\'ee d'Italie, 69007 Lyon, France}
\email{wieslawa.niziol@ens-lyon.fr} 
\date{\today}
\thanks{The second author's research was supported in part by   the grant ANR-14-CE25.}
 \maketitle 
 \begin{abstract} We  prove a mixed characteristic analog of the Beilinson-Lichtenbaum Conjecture for $p$-adic motivic cohomology.    It gives a description, in the stable range,  of $p$-adic motivic cohomology (defined using algebraic cycles) in terms of differential forms. This generalizes a result of Geisser  \cite{GD} from small Tate twists to all twists and uses as a critical new ingredient  the comparison theorem between syntomic complexes and $p$-adic nearby cycles proved recently in \cite{CN}.
    \end{abstract}
 \tableofcontents
 \section{Introduction} 
 For a smooth variety over a field of characteristic zero, the Beilinson-Lichtenbaum Conjecture states that, in a certain stable range,  the $p$-adic motivic cohomology is equal to the \'etale cohomology:
 $$
 H^i_{\M}(X,\Z/p^n(r))\stackrel{\sim}{\to} H^i_{\eet}(X,\Z/p^n(r)), \quad i\leq r.
 $$
 Here motivic cohomology is defined as the hypercohomology of  the   Bloch's cycle complex $\Z/p^n(r)_{\M}$. This conjecture  follows \cite{SV} from the Bloch-Kato Conjecture that was proved by Voevodsky and Rost \cite{Wei}.

 For a smooth variety over a field of positive characteristic $p$, the analog of the Beilinson-Lichtenbaum Conjecture states that, in the same stable range,  the $p$-adic motivic cohomology is equal to the logarithmic de Rham-Witt cohomology:
 $$
  H^i_{\M}(X,\Z/p^n(r))\stackrel{\sim}{\to} H^{i-r}_{\eet}(X,W_n\Omega^{r}_{X,\log}).
 $$
 It was proved by Geisser-Levine \cite{GL}.
 
 The purpose of this note is to prove a mixed characteristic analog of the Beilinson-Lichtenbaum Conjecture for $p$-adic motivic cohomology. 
 Let $X$ be a semistable  scheme over $\so_K$ -- a complete discrete valuation ring with fraction field
$K$  of characteristic 0 and with perfect
residue field $k$ of  characteristic $p$. We fix a uniformizer $\varpi$ of $K$. Let $F$ be the fraction field of the ring of Witt vectors $W(k)$. We assume that the special fiber $X_0$ of $X$ is smooth and treat $X$ as a log-scheme. We show that, in the same stable range as above, the $p$-adic motivic cohomology   of $X_{\tr}$ --  the open set where  the log-structure  is trivial -- is equal to  the (logarithmic)  syntomic-\'etale cohomology of $X$. This relates algebraic cycles to differential forms.

\begin{corollary} We have the following natural isomorphism\footnote{For a smooth scheme $Y$, we set  $H^*_{\M}(Y ,\Q_p(r)):=H^*\holim_n\R \Gamma(Y_{\Zar},\Z/p^n(r)_M)\otimes\Q$. }
 $$
 H^i_{\M}(X_{\tr},\Q_p(r))\stackrel{\sim}{\to} H^{i}_{\eet}(X,\se(r))_{\Q}, \quad i\leq r,
 $$
 where $\se(\cdot)$ denotes the  syntomic-\'etale cohomology complex.
 If $X$ is proper, this yields  the following natural  isomorphism
 $$
  H^i_{\M}(X_{\tr},\Q_p(r))\stackrel{\sim}{\to} H^{i}_{\eet}(X,\sss(r))_{\Q}, \quad i\leq r,
  $$
  where $\sss(\cdot)$ denotes the  syntomic cohomology complex.
   \end{corollary}
The rational syntomic cohomology  $H^{*}_{\eet}(X,\sss(r))_{\Q}$ above is that defined in \cite{FM} as filtered Frobenius eigenspace of  crystalline cohomology\footnote{It differs from the one defined in \cite{NN} by the absence of log-structure associated to the special fiber.}. We show in the appendix
that it is isomorphic to the logarithmic version of  the  convergent syntomic cohomology  defined in  \cite{N1} as well as to the  rigid syntomic cohomology defined in  \cite{Be1, Gr}. 

  The above corollary  is a simple consequence of the following theorem which is the main result of this paper.
\begin{theorem}
\label{main0}Let $r\geq 0$. Let $j^{\prime}_*:X_{\tr}\to X$ be the natural open immersion. 
Then there are natural cycle class maps between complexes of sheaves on the Nisnevich site of $X$ and $X_0$, respectively, 
$$
\cl^{\synt}_r:\R j^{\prime}_*\Z/p^n(r)_{\M}\to \se^{\prime}_n(r)_{\Nis},\quad \cl^{\synt}_r:i^* \R j^{\prime}_*\Z/p^n(r)_{\M}\to \sss^{\prime}_n(r)_{\Nis},
$$
where $i:X_0\hookrightarrow X$ is the special fiber of $X$. They are compatible with the \'etale cycle class maps and 
are $p^N$-quasi-isomorphisms, i.e., 
 the kernels and cokernels of the maps induced on the cohomology sheaves are annihilated by $p^N$ for a constant   
$N=N(e,p,r)$, which depends on the absolute ramification index $e$ of $K$, $r$, but not on $X$ or $n$\footnote{If $K$ has enough roots of unity then  $N={N^{\prime}r}$ for a universal constant $N^{\prime}$ {\rm (not depending on $p$, $X$, $K$, $n$ or $r$)}. See Section (2.1.1) of \cite{CN} for what it means for a field to contain enough roots of unity. The field $F$ contains enough roots of unity and for any $K$, the field $K(\zeta_{p^n})$, for $n\geq c(K)+3$, where $c(K)$ is  the conductor of $K$, contains enough roots of unity. }.
\end{theorem}
The syntomic-\'etale cohomology $\se^{\prime}_n(r)$  was defined by Fontaine-Messing \cite{FM} by gluing syntomic cohomology $\sss^{\prime}_n(r)$ on $X_0$ 
with \'etale cohomology on the generic fiber via the relative fundamental exact sequence of $p$-adic Hodge Theory. It is a complex of sheaves on the \'etale site of $X$. We extend this definition to logarithmic schemes (where one replaces syntomic cohomology by logarithmic syntomic cohomology). The Nisnevich version that appears in the above theorem is defined by projecting to the Nisnevich site and truncating at $r$:
$$\se^{\prime}_n(r)_{\Nis}:=\tau_{\leq r}\R\varepsilon_*\se^{\prime}_n(r),\quad \sss^{\prime}_n(r)_{\Nis}:=\tau_{\leq r}\R\varepsilon_*\sss^{\prime}_n(r),
$$
  where $\varepsilon: X_{\eet}\to X_{\Nis}$ is the natural projection. 

  The syntomic part of the above theorem (hence of the above corollary as well), for twists $r\leq p-2$ (where no constants are needed) was proved by Geisser\footnote{Geisser's result was conditional on the Bloch-Kato Conjecture which at the time of the publication of  his paper was not a theorem yet.} \cite[Theorem 1.3]{GD}. The key ingredient in his proof is the exact sequence of Kurihara
  \cite{Kur} that links syntomic cohomology with $p$-adic nearby cycles coupled with the Beilinson-Lichtenbaum Conjecture over fields of characteristic zero and $p$.
  Our proof of Theorem~\ref{main0} proceeds in a similar manner using as the main new  ingredient  the relation between syntomic complexes and $p$-adic nearby cycles proved recently in \cite{CN}. 
  
    We will now describe it briefly in the case when there is no horizontal log-structure.
  First, we show that we have the $p^{Nr}$-distinguished triangle (on the \'etale site of $X_0$), for a universal constant $N$,
  \begin{equation}
\label{seq0}
\se^{\prime}_n(r)_X\to \se^{\prime}_n(r)_{X^{\times}}\to W_n\Omega^{r-1}_{X_0,\log}[-r],
\end{equation}
where $W_n\Omega^{r-1}_{X_0,\log}[-r]$ denotes the logarithmic de Rham-Witt sheaf and $X^{\times}$ denotes the scheme $X$ with added  log-structure coming  from the special fiber. 
The syntomic-\'etale cohomology $\se^{\prime}_n(r)_{X^{\times}}$ comes equipped with a period map
  $$\alpha_r: \se^{\prime}_n(r)_{X^{\times}}\to \R j_*\Z/p^n(r)^{\prime}_{X_K},
    $$
    where $j_*:X_K\hookrightarrow X$ and $\Z/p^n(r)^{\prime}=(p^aa!)^{-1}{\mathbf Z}/p^n(r)$ for $r=(p-1)a+b,a,b\in{\mathbf Z}, 0\leq b < p-1$ . 
    Projecting it to the Nisnevich site and truncating at $r$ we obtain the Nisnevich syntomic-\'etale period map
    $$\alpha_r: \se^{\prime}_n(r)_{{X^{\times}},\Nis}\to \tau_{\leq r}\R \varepsilon_* \R j_*\Z/p^n(r)^{\prime}_{X_K}.
        $$
      
        The computations of $p$-adic nearby cycles via syntomic cohomology from \cite{CN} imply  that this is a $p^N$-quasi-isomorphism, for a constant $N$ as in the theorem. Hence, from (\ref{seq0}),  we obtain the $p^N$-distinguished triangle
   \begin{equation}
   \label{seq1}
        \se^{\prime}_n(r)_{{X},\Nis}\lomapr{\alpha_r} \tau_{\leq r}\R j_* \tau_{\leq r}\R \varepsilon_* \Z/p^n(r)^{\prime}_{X_K}\to i_*W_n\Omega^{r-1}_{X_0,\log}[-r].
     \end{equation}
  
     Next, we note that the localization sequence in motivic cohomology yields the following distinguished triangle (on the Nisnevich site of $X$)
     $$
     \Z/p^n(r)_{\M}\to j_*\Z/p^n(r)_{\M}\to i_*\Z/p^n(r-1)_{\M}[-1].
     $$     
     By the Beilinson-Lichtenbaum Conjecture and the computations of Geisser-Levine \cite{GL} of motivic cohomology in characteristic $p$, we have the cycle class map quasi-isomorphisms
     $$
     \Z/p^n(r)_{\M}\stackrel{\sim}{\to} \tau_{\leq r}\R\varepsilon_*\Z/p^n(r)_{X_K},\quad      \Z/p^n(r)_{\M}\stackrel{\sim}{\to}      W_n\Omega^{r}_{X_0,\log}[-r] . 
        $$
        The above triangle  becomes
       \begin{equation}
       \label{seq2}
 \Z/p^n(r)_{\M} \to  j_*\tau_{\leq r}\R\varepsilon_*\Z/p^n(r)_{X_K} \to i_*W_n\Omega^{r-1}_{X_0,\log}[-r]        
 \end{equation}
  Since $ j_*\Z/p^n(r)_{\M}\stackrel{\sim}{\to} \R j_*\Z/p^n(r)_{\M}$,  $\tau_{\leq r}\Z/p^n(r)_{\M}\stackrel{\sim}{\to} \Z/p^n(r)_{\M}$,  the cycle class map of Theorem \ref{main0} can now be  obtained by comparing sequences (\ref{seq1}) and (\ref{seq2}).

\begin{acknowledgments} We would like to thank Grzegorz Banaszak, Pierre Colmez, and Bruno Kahn for many discussions related to the content of this paper. 
 \end{acknowledgments}
 \subsubsection{Notation and Conventions}
 We assume all the schemes   to be locally noetherian. We work in the category of fine log-schemes. 
\begin{definition}
\label{1saint}
Let $N\in {\mathbf N}$. For a morphism $f: M\to M^{\prime}$ of ${\mathbf Z}_p$-modules, we say that $f$ is 
{\it $p^N$-injective} (resp. {\it $p^N$-surjective}) if its kernel (resp. its cockernel) is annihilated by $p^N$ 
and we say that $f$ is {\it $p^N$-isomorphism} if it is $p^N$-injective and $p^N$-surjective. 
We define in the same way the notion of {\it $p^N$-distinguished triangle} or {\it $p^N$-acyclic complex} 
(a complex whose cohomology groups are annihilated by $p^N$) as well as the notion of {\it $p^N$-quasi-isomorphism}
 (map in the derived category that induces a $p^N$-isomorphism on cohomology). 
\end{definition}
We will use a shorthand for certain homotopy limits. Namely,  if $f:C\to C'$ is a map  in the dg derived category of abelian groups, we set
$$[\xymatrix{C\ar[r]^f&C'}]:=\holim(C\to C^{\prime}\leftarrow 0).$$ 
And we set
$$
\left[\begin{aligned}
\xymatrix{C_1\ar[d]\ar[r]^f & C_2\ar[d]\\
C_3\ar[r]^g & C_4
}\end{aligned}\right]
:=[[C_1\stackrel{f}{\to} C_2]\to [C_3\stackrel{g}{\to} C_4]],
$$ 
for 
 a commutative diagram (the one inside the large bracket) in the dg derived category of abelian groups.
\section{Syntomic cohomology}Let $\so_K$ be a complete discrete valuation ring with fraction field
$K$  of characteristic 0 and with perfect
residue field $k$ of characteristic $p$. Let $\varpi$ be a uniformizer of $\so_K$; we will keep it fixed throughout the paper\footnote{This is necessary to fix an embedding of $\Spec(\so_K)$ into a smooth scheme over $\Z_p$.}. Let
$W(k)$ be the ring of Witt vectors of $k$ with 
 fraction field $F$ (i.e, $W(k)=\so_F$); let $e$ be the ramification index of $K$ over $F$.  Let $\sigma=\phi$ be the absolute
Frobenius on $W(\overline {k})$. 
For a $\so_K$-scheme $X$, let $X_0$ denote
the special fiber of $X$ and let $X_n$ denote the reduction modulo $p^n$ of $X$. We will denote by $\so_K$, 
${\so_K}^{\times}$, and ${\so_K}^0$ the scheme $\Spec ({\so_K})$ with the trivial, canonical
(i.e., associated to the closed point), and $({\mathbf N}\to {\so_K}, 1\mapsto 0)$ 
log-structure respectively. 
 
  In this section we will briefly review the definitions of syntomic and syntomic-\'etale cohomologies and their basic properties.
 We refer the reader for details to \cite[2]{Ts}.
\subsection{Syntomic cohomology}
 For a log-scheme $X$ we denote by $X_{\synt}$ the small syntomic site of $X$. It is 
 built from log-syntomic morphisms $f:Y\to Z$ in the sense of Kato \cite[2.5]{K1} (see also \cite[6.1]{BM}), i.e., 
the morphism  $f$ is integral, the underlying morphism of schemes is flat and locally of finite presentation, 
and, \'etale locally on $Y$, there is a factorization $Y\stackrel{i}{\hookrightarrow}W\stackrel{h}{\rightarrow}Z$ 
where $h$ is log-smooth and $i$ is an exact closed immersion that is transversally regular over $Z$.

 For a log-scheme $X$ log-syntomic over $\Spec(W(k))$, define 
$$
\so^{\crr}_n(X) =H^0_{\crr}(X_n,\so_{X_n}),\qquad 
\sj_n^{[r]}(X) =H^0_{\crr}(X_n,\sj^{[r]}_{X_n}),
$$
where $\so_{X_n}$ is the structure sheaf of the absolute crystalline site (i.e., over $W_n(k)$), 
$\sj_{X_n}=\kker(\so_{X_n/W_n(k)}\to \so_{X_n})$, and $\sj^{[r]}_{X_n}$ is its $r$'th divided power of $\sj_{X_n}$.
Set $\sj^{[r]}_{X_n}=\so_{X_n}$ if $r\leq 0$.
We know \cite[II.1.3]{FM} that the presheaves $\sj_n^{[r]} $ are sheaves on $X_{n,\synt}$, flat over ${\mathbf Z}/p^n$, and that
$\sj^{[r]}_{n+1}\otimes{\mathbf Z}/p^n\simeq  \sj^{[r]}_{n}$.  There is a  natural, compatible with Frobenius,  and 
functorial isomorphism 
$$
H^*(X_{\synt},\sj_n^{[r]})\simeq H^*_{\crr}(X_n,\sj^{[r]}_{X_n}).
$$
It is easy to see that $\phi(\sj_n^{[r]} )\subset p^r\so^{\crr}_n$ for $0\leq r\leq p-1$. This fails in general and we modify
$\sj_n^{[r]}$:  $$\sj_n^{<r>}:= \{x\in \sj_{n+s}^{[r]}\mid \phi(x)\in p^r\so^{\crr}_{n+s}\}/p^n ,$$
for some $s\geq r$.
This definition is independent of $s$. We check that $\sj_n^{<r>}$  is flat over ${\mathbf Z}/p^n$ and
 $\sj^{<r>}_{n+1}\otimes{\mathbf Z}/p^n\simeq  \sj^{<r>}_{n}$.
This allows us to define the divided  Frobenius $\phi_r="\phi/p^r": \sj_n^{<r>} \to \so^{\crr}_n$.

   Set $$\sss_n(r):=\Cone(\sj_n^{<r>} \stackrel{1-\phi_r}{\longrightarrow}\so^{\crr}_n)[-1].$$
Since the following sequence is exact
$$
\begin{CD}
0@>>> \sss_n(r)@>>> \sj_n^{<r>} @>1-\phi_r >>\so^{\crr}_n @>>> 0,
\end{CD}
$$
we  actually have 
$$\sss_n(r):=\kker(\sj_n^{<r>} \stackrel{1-\phi_r}{\longrightarrow}\so^{\crr}_n).$$
    In the same way we can define syntomic sheaves $\sss_n(r)$ on $X_{m,\synt}$ for $m\geq n$. 
Abusing notation, we set $\sss_n(r)=i_*\sss_n(r)$ for the natural map $i: X_{m,\synt}\to X_{\synt}$. Since $i_*$ is exact, $H^*(X_{m,\synt},\sss_n(r))=H^*(X_{\synt},\sss_n(r))$.
Because of that we will write $\sss_n(r)$ for the syntomic sheaves on $X_{m,\synt}$ as well as on $X_{\synt}$. We will also need the "undivided" version of syntomic complexes of sheaves:
$$\sss'_n(r):=\Cone(\sj_n^{[r]} \stackrel{p^r-\phi}{\longrightarrow}\so^{\crr}_n)[-1].$$
For $r,i\geq 0$, we have the long exact sequences
\begin{align}
\label{exact} \rightarrow
H^i(X_{\eet},\sss_n(r)) &  \rightarrow
H^i_{\crr}(X_{n},J^{<r>}_{X_{n}}) \stackrel{
1-\phi_r}{\longrightarrow}  H^i_{\crr}(X_{n},
\so_{X_{n}})
\rightarrow\\
 \rightarrow
H^i(X_{\eet},\sss^{\prime}_n(r)) &  \rightarrow
H^i_{\crr}(X_{n},J^{[r]}_{X_{n}}) \stackrel{
p^r-\phi}{\longrightarrow}  H^i_{\crr}(X_{n},
\so_{X_{n}})
\rightarrow\notag
\end{align}
\begin{proposition}(\cite[Prop. 3.12]{CN})
For $X$ a fine and saturated log-smooth log-scheme over $\so_K$ and $0\leq r\leq p-2$, the natural map of complexes of sheaves on the \'etale site of $X_0$  
    $$
    \tau_{\leq r}\sss_n(r)\to \sss_n(r)
$$
is a quasi-isomorphism.
For $X$ semistable over $\so_K$\footnote{A scheme $X$ over $\so_K$ is called semistable if it is surjective on $\Spec \so_K$, regular, and there is a distinguished divisor "at infinity" $D_{\infty}$ which  is a strict relative normal crossing divisor and which together with the special fiber forms a strict normal crossing divisor.} and $r\geq 0$, the  natural map of complexes of sheaves on the \'etale site of $X_0$  
    $$
    \tau_{\leq r}\sss^{\prime}_n(r)\to \sss^{\prime}_n(r)
$$
is  a $p^{Nr}$-quasi-isomorphism for a universal constant $N$. 
\end{proposition}

 The natural map $\omega: \sss^{\prime}_n(r)\to \sss_n(r)$ induced by the maps $p^{r}: \sj_n^{[r]}\to \sj_n^{<r>}$ and $\id : \so^{\crr}_n \to \so^{\crr}_n $ has kernel and cokernel  killed by $p^{r}$. 
 So does the map $\tau: \sss_n(r)\to \sss^{\prime}_n(r)$ induced by the maps $\id : \sj_n^{<r>}\to \sj_n^{[r]}$ and $p^{r} : \so^{\crr}_n \to \so^{\crr}_n $. We have $\tau\omega=\omega\tau=p^{r}$.

If it does not cause confusion, we will  write $\sss_n(r)$, $\sss^{\prime}_n(r)$ also  for $\R\ve_*\sss_n(r)$, $\R\ve_*\sss^{\prime}_n(r)$, respectively, where $\varepsilon: X_{n,\synt}\to X_{n,\eet}$ is the natural projection to the \'etale site (or sometimes to the Nisnevich site)

\subsubsection{Syntomic cohomology and differential forms}
Let $X$ be a
syntomic  scheme over $W(k)$. Recall the
differential definition
\cite{K} of syntomic cohomology. Assume first
that we have an immersion $\iota:X\hookrightarrow Z$ over $W(k)$ such that
$Z$ is a smooth $W(k)$-scheme endowed with a compatible system of liftings of the Frobenius
$\{F_n:Z_n\to Z_n\}$. Let $D_n=D_{X_n}(Z_n)$ be the PD-envelope of $X_n$ in $Z_n$ (compatible with the canonical PD-structure on $pW_n(k)$) and
$J_{D_n}$ the ideal of $X_n$ in $D_n$. Set $J^{<r>}_{D_n}:=\{a\in J_{D_{n+s}}^{[r]}|\phi(a)\in p^r\so_{D_{n+s}}\}/p^n$ for some $s\geq r$.
For $0\leq r\leq p-1$, $J_{D_n}^{<r>}=J_{D_n}^{[r]}$.
 This definition is independent of $s$.
Consider the following
complexes of sheaves on $X_{\eet}$. 
\begin{align}
 \label{lifted}
S_n(r)_{X,Z}: & =\Cone(J_{D_n}^{<r-{\scriptscriptstyle\bullet}>}\otimes
\Omega\kr_{Z_n}\stackrel{ 1-\phi_r}{\longrightarrow} \so_{D_n}\otimes
\Omega\kr_{Z_n})[-1],\\
S^{\prime}_n(r)_{X,Z}: & =\Cone(J_{D_n}^{[r-{\scriptscriptstyle\bullet}]}\otimes
\Omega\kr_{Z_n}\verylomapr{p^{r} -\phi}\so_{D_n}\otimes
\Omega\kr_{Z_n})[-1],\notag
\end{align}
where $\Omega\kr_{Z_n}:=\Omega\kr_{Z_n/W_n(k)}$ and $\phi_r $ is $"\phi/p^r"$ (see \cite[2.1]{Ts} for details). The complexes
$S_n(r)_{X,Z}$, $S^{\prime}_n(r)_{X,Z}$ are, up to canonical quasi-isomorphisms, independent of
the choice of $\iota$ and $\{F_n\}$ (and we will omit the subscript $Z$ from the notation). Again, the natural maps $\omega: S^{\prime}_n(r)_X \to S_n(r)_X$ and 
$\tau: S_n(r)_X \to S^{\prime}_n(r)_X$ have kernels and cokernels annihilated by $p^{r}$.

  In general, immersions as above exist \'etale locally, and we define
$S_n(r)_X\in
{\mathbf D}^+(X_{\eet},{\mathbf Z}/p^n)$ by gluing the local complexes. We define
 $S^{\prime}_n(r)_X$ in a similar way.

 Let now $X$ be a log-syntomic scheme over $W(k)$.
Using log-crystalline cohomology,  the above construction of syntomic complexes goes through almost
verbatim (see \cite[2.1]{Ts} for details) to yield
the logarithmic analogs $S_n(r)$ and  $S^{\prime}_n(r)$ on $X_{\eet}$. 
In this paper we are often interested in log-schemes coming from a regular syntomic scheme $X$ over $W(k)$ and a relative simple (i.e., with no self-intersections) normal crossing divisor $D$  on $X$. In such cases we will write $S_n(r)_X( D)$ and $S^{\prime}_n(r)_X(D)$ for the  syntomic complexes and use the Zariski topology instead of the \'etale one.  

\subsubsection{Products}  We need to discuss products. Assume that we are in the lifted situation (\ref{lifted}). Then we have a product structure
$$\cup: S^{\prime}_n(r)_{X,Z}\otimes S^{\prime}_n(r^{\prime})_{X,Z}\to S^{\prime}_n(r+r^{\prime})_{X,Z}, \quad r,r^{\prime}\geq 0,
$$
defined by the following formulas
\begin{align*}
(x,y)\otimes (x^{\prime},y^{\prime}) & \mapsto (xx^{\prime}, (-1)^ap^{r}xy^{\prime}+y\phi(x^{\prime}))\\
(x,y)\in S^{\prime}_n(r)_{X,Z}^a & =(J_{D_n}^{[r-a]}\otimes
\Omega^a_{Z_n})\oplus(\so_{D_n}\otimes
\Omega^{a-1}_{Z_n}),\\
(x^{\prime},y^{\prime})\in S^{\prime}_n(r^{\prime})_{X,Z}^b & = (J_{D_n}^{[r^{\prime}-b]}\otimes
\Omega^b_{Z_n})\oplus(\so_{D_n}\otimes
\Omega^{b-1}_{Z_n}).
\end{align*}
Globalizing, we obtain the product structure
$$\cup: S^{\prime}_n(r)_{X}\otimes ^{{\mathbb L}}S^{\prime}_n(r^{\prime})_{X}\to S^{\prime}_n(r+r^{\prime})_{X}, \quad r,r^{\prime}\geq 0.
$$
This product is clearly compatible with the crystalline product.

   Similarly, we have the product structures
$$\cup: S_n(r)_{X,Z}\otimes S_n(r^{\prime})_{X,Z}\to S_n(r+r^{\prime})_{X,Z}, \quad r,r^{\prime}\geq 0,
$$
defined by the formulas
\begin{align*}
(x,y)\otimes (x^{\prime},y^{\prime}) & \mapsto (xx^{\prime}, (-1)^axy^{\prime}+y\phi_{r^{\prime}}(x^{\prime}))\\
(x,y)\in S_n(r)_{X,Z}^a & =(J_{D_n}^{<r-a>}\otimes
\Omega^a_{Z_n})\oplus(\so_{D_n}\otimes
\Omega^{a-1}_{Z_n}),\\
(x^{\prime},y^{\prime})\in S_n(r^{\prime})_{X,Z}^b & = (J_{D_n}^{<r^{\prime}-b>}\otimes
\Omega^b_{Z_n})\oplus(\so_{D_n}\otimes
\Omega^{b-1}_{Z_n}).
\end{align*}
Globalizing, we obtain the product structure
$$\cup: S_n(r)_{X}\otimes ^{{\mathbb L}}S_n(r^{\prime})_{X}\to S_n(r+r^{\prime})_{X}, \quad r,r^{\prime}\geq 0.
$$
This product is also clearly compatible with the crystalline product.
  
 The above product structures are compatible with the maps $\omega$.
On the other hand the maps $\tau$ are, in general, not compatible with products.

\subsubsection{Symbol maps}   Let $X$ be a regular syntomic scheme over $W(k)$ with a divisor $D$ with relative simple normal crossings.
Recall that there are symbol  maps defined by Kato and Tsuji \cite[2.2]{Ts}
\begin{equation}
\label{symbols}
(M^{\gp}_{X,n})^{\otimes r}\to H^r(S^{\prime}_n(r)_X(D)),\quad (M^{\gp}_{X,n+1})^{\otimes r}\to H^r(S_n(r)_X(D)),\quad r\geq 0, 
\end{equation}
 where, for a log-scheme $X$, $M_X$ denotes its log-structure.
For $r=1$, we get the first Chern class maps (recall that $M^{\gp}_X=j_*\so^*_{X\setminus D}$, where  $j:X\setminus D\hookrightarrow X$ is the natural immersion)
\begin{align*}
{c}_1^{\synt}:\quad & j_*\so^*_{X\setminus D}[-1]\to i_*j_*\so^*_{{(X\setminus D)}_{n+1}}[-1]\to
S_n(1)_X(D),\\
{c}_1^{\synt}:\quad & j_*\so^*_{X\setminus D}[-1]\to i_*j_*\so^*_{{(X\setminus D)}_{n}}[-1]\to
S^{\prime}_n(1)_X(D),
\end{align*}
that are compatible, i.e., the following diagram commutes
$$
\xymatrix{
j_*\so^*_{X\setminus D}[-1] \ar[d]^{p{c}_1^{\synt}} \ar[r]^{{c}_1^{\synt}}  & S^{\prime}_n(1)_X(D)\ar[dl]^{\omega}\\
S_n(1)_X(D) &\\
}
$$
In the lifted situation these classes are defined in the following way. Let $C_n$ be the complex
$$
(1+J_{D_n}\to M^{\gp}_{D_n})\simeq j_*\so^*_{(X\setminus D)_n}[-1].
$$
 The Chern class maps
\begin{equation*}
{c}_1^{\synt}:j_*\so^*_{{(X\setminus D)}_{n}}[-1]\to
S^{\prime}_n(1)_X(D), \quad
{c}_1^{\synt}: j_*\so^*_{{(X\setminus D)}_{n+1}}[-1]\to
S_n(1)_X(D),
\end{equation*}
 are defined by the morphisms of complexes $$C_n\to S^{\prime}_n(1)_{X,Z},\quad
C_{n+1}\to S_n(1)_{X,Z}$$ given by the formulas
\begin{align*}
1+J_{D_n}\to (S^{\prime}_n(1)_{X,Z})^0 & =J_{D_n};\quad a\mapsto  \log a;\\
1+J_{D_{n+1}}\to (S_n(1)_{X,Z})^0 & =J_{D_n};\quad a\mapsto  \log a \mod p^n;
\end{align*}
and
\begin{align*}
 M^{\gp}_{D_n}\to (S^{\prime}_n(1)_{X,Z})^1 & =(\so_{D_n}\otimes \Omega^1_{Z_n})\oplus \so_{D_n};
\quad b\mapsto (d \log b,  \log (b^p\phi_{D_n}(b)^{-1}));\\
M^{\gp}_{D_{n+1}}\to (S_n(1)_{X,Z})^1 & =(\so_{D_n}\otimes \Omega^1_{Z_n})\oplus \so_{D_n};
\quad b\mapsto (d\log  b\mod p^n, p^{-1} \log(b^p\phi_{D_{n+1}}(b)^{-1})).\\
\end{align*}
The symbol maps (\ref{symbols}) for general $r$ are obtained from $r=1$ using the product structure on syntomic cohomology.
\subsection{Syntomic-\'etale cohomology}
We will now recall the definition and basic properties of
syntomic-\'etale cohomology.  The relationship between syntomic cohomology and syntomic-\'etale cohomology mirrors the one between \'etale nearby cycles and \'etale cohomology. Let $X$ be a log-scheme, log-syntomic  over $\Spec(W(k))$.
  We will  need the logarithmic version of the syntomic-\'etale site of Fontaine-Messing \cite{FM}.
We say that a morphism $\szz\to\sy$ of $p$-adic formal log-schemes over $\Spf (W(k))$ is (small) log-syntomic if every $Z_n\to Y_n$ is (small) log-syntomic.
 For a formal log-scheme $\szz$ the syntomic-\'etale site $\szz_{\synee}$ is defined by taking as objects morphisms $f: \sy\to\szz$ that are small log-syntomic and have  log-\'etale generic fiber. 
 This last condition means that,  \'etale locally on $\sy$,  $f$ has a factorization $\sy\stackrel{i}{\to}\sx\stackrel{g}{\to}\szz$
 with $\sx$ affine, $i$ an exact closed immersion, and $g$ log-smooth such that the map 
 $F\otimes_{W(k)}\Gamma(\sy,I/I^2)\to F\otimes_{W(k)}\Gamma(\sy,i^*\Omega^1_{\sx/\szz})$ is an isomorphism, where $I$ is the ideal of $\so_{\sx}$ defining $\sy$.
  For a log-scheme $Z$, we also have the syntomic-\'etale site $Z_{\synee}$. Here the objects are morphisms $U\to Z$ that are small log-syntomic  with the generic fiber
 $U_K$  log-\'etale over $Z_K$.
 
   Let $\wh{X}$ be the $p$-adic completion of $X$. Let $i:X_{n,\eet}\to X_{\eet}$ and $j:X_{\tr,K,\eet}\to X_{\eet}$ be the natural maps. 
Here $X_{\tr}$ is the open set of $X$ where the log-structure is trivial. We have  the following commutative diagram of maps of topoi
$$
\begin{CD}
\wh{X}_{\synee} @>i_{\synee} >> X_{\synee} @<j_{\synee} << X_{K,\synee}\\
@V\wh{\ve} VV @V\ve VV @V\ve_K VV\\
\wh{X}_{\eet} @>i_{\eet} >> X_{\eet} @< j_{\eet} << X_{K,\eet}
\end{CD}
$$
 Assume first that $0\leq r\leq p-2$. Abusively, let $\sss_n(r)$ denote also  the direct image of $\sss_n(r)$ under the canonical morphism $X_{n,\synt}\to \wh{X}_{\synee}$. 
 By \cite[III.5]{FM},
 for $j^{\prime}: X_{\tr,K,\eet}\to X_{K,\synee}$,
there is  a canonical homomorphism
 $$\alpha_r:
 \sss_n(r) \rightarrow i^*_{\synee}j_{\synee *}j^{\prime}_*G{\mathbf Z}/p^n(r),
 $$
 where $G$ denotes the Godement resolution of a sheaf (or a complex of sheaves). 
 Similarly, for any $r\geq 0$, we get a natural map
$$
\tilde{\alpha}_{r}: \sss_n(r) \rightarrow i^*_{\synee}j_{\synee *}j^{\prime}_*G{\mathbf Z}/p^n(r)',
$$
where ${\mathbf Z}/p^n(r)'=(p^aa!)^{-1}{\mathbf Z}/p^n(r)$ for $r=(p-1)a+b,a,b\in{\mathbf Z}, 0\leq b < p-1$ \cite[III.5]{FM}. Composing with the map 
$\sss'_n(r)\to \sss_n(r)$ we get a natural
 morphism
$$
\alpha_{r}: \sss'_n(r) \rightarrow i^*_{\synee}j_{\synee *}j^{\prime}_*G{\mathbf Z}/p^n(r)'.
$$
\subsubsection{Syntomic complexes and $p$-adic nearby cycles}
  For log-schemes over $\so_K^{\times}$, in a stable range, syntomic cohomology tends to  compute (via the period morphism) $p$-adic nearby cycles . We will briefly recall the relevant theorems. 
 For $0\leq r\leq p-2$, there is a natural homomorphism on the \'etale site of $X_n$
$$
\alpha_{r}: \sss_n(r) \rightarrow i^*\R j_*{\mathbf Z}/p^n(r).
$$  
To define it, 
 we apply $\R \wh{\varepsilon}_*$ to the  map 
$\sss_n(r)\rightarrow i^*_{\synee}\R j_{\synee *}\R j_*^{\prime}{\mathbf Z}/p^n(r)$ induced from the map $\alpha_r$ described above and get 
$$\R \varepsilon_*\sss_n(r) =\R \wh{\varepsilon}_*\sss_n(r)\rightarrow 
\R \wh{\varepsilon}_*i^*_{\synee}\R j_{\synee *}\R j^{\prime}_{*} {\mathbf Z}/p^n(r)=i^*_{\eet}\R \varepsilon_*\R j_{\synee *}\R j^{\prime}_{*}{\mathbf Z}/p^n(r)=i^*\R j_*{\mathbf Z}/p^n(r).
$$
The first equality follows from the fact that the morphism $X_{n,\synt}\to \wh{X}_{\synee}$ is exact \cite[III.4.1]{FM}.
The second equality was proved in \cite[2.5]{KM}, \cite[5.2.3]{Ts0}.
One checks that $\alpha_r$ is compatible with products.
\begin{theorem}(\cite[Theorem 5.1]{Ts1})
\label{input0}
 For $i\leq r\leq p-2$ and  for  a fine and saturated log-scheme $X$ log-smooth over  $\so_K^{\times}$  the period map
\begin{equation}
\label{period2}\alpha_{r}:\quad  \sss_n(r)_{X} \stackrel{\sim}{\rightarrow} \tau_{\leq r}i^*\R j_*{\mathbf Z}/p^n(r)_{X_{\tr}}.
\end{equation}
is an isomorphism. 
\end{theorem}

 Similarly, for any $r\geq 0$, we get a natural map
$$
\tilde{\alpha}_{r}: \sss_n(r) \rightarrow i^*\R j_*{\mathbf Z}/p^n(r)'.
$$
 Composing with the map 
$\omega: \sss'_n(r)\to \sss_n(r)$ we get a natural, compatible with products,  
 morphism
$$
\alpha_{r}: \sss^{\prime}_n(r) \rightarrow i^*\R j_*{\mathbf Z}/p^n(r)'.
$$
\begin{theorem}(\cite[Theorem 1.1]{CN})
\label{input1}
 For   $0\leq i\leq r$ and for a semistable scheme $X$ over $\so_K$,   consider the period map 
\begin{equation}
\label{maineq1}
\alpha_{r}:\quad  \sh^i(\sss^{\prime}_n(r)_{X}) \rightarrow i^*\R^ij_*{\mathbf Z}/p^n(r)'_{X_{\tr}}.
\end{equation}
If $K$ has enough roots of unity then the kernel  and cokernel of this map are annihilated by $p^{Nr}$ for a universal constant $N$ {\rm (not depending on $p$, $X$, $K$, $n$ or $r$)}.
In general, the kernel  and cokernel of this map are annihilated by $p^N$ for an integer $N=N(e,p,r)$, which depends on $e$, $r$, but not on $X$ or $n$.
   \end{theorem}
   \subsubsection{Syntomic-\'etale cohomology}
   Recall \cite[III.4.4]{FM}, \cite[5.2.2]{Ts0} that the functor 
$\sff\mapsto (i^*_{\synee}\sff,j^*_{\synee}\sff,i^*_{\synee}\sff\to i^*_{\synee}j_{\synee*}j^*_{\synee}\sff)$ from the category of sheaves 
on $X_{\synee}$ to the category of triples $(\sg,\sh,\sg\to i^*_{\synee}j_{\synee*}\sh)$, where $\sg$ (resp. $\sh$) are sheaves on $\wh{X}_{\synee}$ 
(resp. $X_{K,\synee}$) is an equivalence of categories. It follows that we can glue the complexes of sheaves $\sss_n(r)$ and $\sss^{\prime}_n(r)$ and the complexes of sheaves  $j_*^{\prime}G{\mathbf Z}/p^n(r)$
 and $j_*^{\prime}G{\mathbf Z}/p^n(r)^{\prime}$ by the maps
$\alpha_r$ and obtain complexes of sheaves $\se_n(r)$ and $\se_n^{\prime}(r)$ on $X_{\synee}$. We have the exact sequences
\begin{align*}
0\to & j_{\syneet!}j^{\prime}_*G{\mathbf Z}/p^n(r)\to \se_n(r)\to i_*\sss_n(r)\to 0,\quad & 0\leq r\leq p-2;\\
0\to & j_{\syneet!}j^{\prime}_*G{\mathbf Z}/p^n(r)^{\prime}\to \se^{\prime}_n(r)\to i_*\sss^{\prime}_n(r)\to 0,\quad & r\geq 0.\\
\end{align*}
\begin{remark}
The syntomic-\'etale complexes $\se_n(r)$ that we described here are the same (in the derived category) as those defined by Fontaine-Messing in \cite[5]{FM} in the case when $X_{\tr}=X$  but differ from those defined by Tsuji in \cite[5.2]{Ts0} in the general situation. More specifically, we have $$\se^T_n(r)=\sh^0(\se_n(r)),
$$
where we wrote $\se^T_n(r)$ for the syntomic-\'etale sheaves of Tsuji.
\end{remark}

   If it does not cause confusion, we will denote by $\se_n(r)$ and $\se^{\prime}_n(r)$  also the derived pushforwards of  $\se_n(r)$ and $\se^{\prime}_n(r)$ to $X_{\eet}$. Notice that they are quasi-isomorphic to 
the complexes obtained by gluing the complexes of sheaves $\sss_n(r) $ and $\sss'_n(r) $ and the complexes of sheaves $j^{\prime}_*G{\mathbf Z}/p^n(r)^{\prime}$  by the maps $\tilde{\alpha}_{r}$ and $\alpha_r$. Hence we have the distinguished triangles
\begin{equation}
\label{Niis}
j_{\eet!} \R j_*^{\prime}{\mathbf Z}/p^n(r)^{\prime}\to \se_n(r)\to i_*\sss_n(r),\quad j_{\eet !}  \R j_*^{\prime}{\mathbf Z}/p^n(r)^{\prime}\to \se^{\prime}_n(r)\to i_*\sss^{\prime}_n(r),
\end{equation}
where  $j^{\prime}: X_{\tr,K}\to X_{K}$,
as well as the natural maps
$$\tilde{\alpha}_r: \se_n(r)\to \R j_*{\mathbf Z}/p^n(r)^{\prime},\quad \alpha_r: \se_n(r)^{\prime}\to \R j_*{\mathbf Z}/p^n(r)^{\prime}
$$
compatible with the maps $\tilde{\alpha}_r$ and $\alpha_r$. For $a\geq 0$, we have the truncated version of the above - the distinguished triangles
\begin{equation}
\label{Niis1}
j_{\eet!} \tau_{\leq a}\R j_*^{\prime}{\mathbf Z}/p^n(r)^{\prime}\to \tau_{\leq a}\se_n(r)\to i_*\tau_{\leq a}\sss_n(r),\quad j_{\eet !} \tau_{\leq a} \R j_*^{\prime}{\mathbf Z}/p^n(r)^{\prime}\to \tau_{\leq a}\se^{\prime}_n(r)\to i_*\tau_{\leq a}\sss^{\prime}_n(r).
\end{equation}
\subsubsection{Syntomic-\'etale cohomology and \'etale cohomology of the generic fiber}
For a log-scheme over $\so_K^{\times}$, in a stable range, syntomic-\'etale cohomology tends to compute \'etale cohomology of the generic fiber.
\begin{theorem}
\label{keylemma11}
Let $X$ be a log-scheme log-smooth   over ${\so_K}^{\times}$. 
Let $j:X_{\tr}\hookrightarrow X$ be the natural open immersion. Then
\begin{enumerate}
\item we have a natural quasi-isomorphism
$$
\tilde{\alpha}_r: \tau_{\leq r}\se_n(r) \simeq \tau_{\leq r}\R j_*{\mathbf Z}/p^n(r), \quad 0\leq r\leq p-2.
$$
\item if $X$ is semistable, there is a constant $N$ as in Theorem \ref{input1} and a natural morphism $$
{\alpha}_r: \se^{\prime}_n(r)\to  
\R j_*{\mathbf Z}/p^n(r)^{\prime},\quad r\geq 0,
$$
such that the induced map on cohomology sheaves in degree $q\leq r$ has kernel and cokernel annihilated by $p^{N}$. 
\end{enumerate}
\end{theorem}
\begin{proof}Assume that $0\leq r\leq p-2$. 
Consider the following commutative diagram of distinguished triangles
$$
\begin{CD}
j_{\eet!} \tau_{\leq r}\R j_*^{\prime}{\mathbf Z}/p^n(r)@>>>  \tau_{\leq r}\se_n(r)@>>> i_*\sss_n(r)\\
@V\wr V \id V @VV\tilde{\alpha}_rV @V\wr V\alpha_r V\\
j_{\eet!} \tau_{\leq r}\R j_*^{\prime}{\mathbf Z}/p^n(r) @>>> \tau_{\leq r}\R j_*{\mathbf Z}/p^n(r) @>>> i_*i^*\tau_{\leq r}\R j_*{\mathbf Z}/p^n(r)
\end{CD}
$$
The top  triangle is distinguished because we have the distinguished triangle from (\ref{Niis1}) and the natural map $\tau_{\leq r}\sss_n(r)\stackrel{\sim}{\to} \sss_n(r)$ is a quasi-isomorphism. The map $\alpha_r$ is a quasi-isomorphism by the main theorem of \cite{Ts1}. The first part of the theorem follows. 

  For the second part consider the following commutative diagram of distinguished triangles
$$
\begin{CD}
j_{\eet!} \tau_{\leq r}\R j_*^{\prime}{\mathbf Z}/p^n(r)^{\prime}@>>>  \tau_{\leq r}\se_n^{\prime}(r)@>>> i_*\tau_{\leq r}\sss^{\prime}_n(r)\\
@V\wr V \id V @VV{\alpha}_rV @VV\alpha_r V\\
j_{\eet!} \tau_{\leq r}\R j_*^{\prime}{\mathbf Z}/p^n(r)^{\prime} @>>> \tau_{\leq r}\R j_*{\mathbf Z}/p^n(r)^{\prime} @>>> i_*i^*\tau_{\leq r}\R j_*{\mathbf Z}/p^n(r)^{\prime}
\end{CD}
$$
By \cite[Theorem 1.1]{CN}, the right period map $\alpha_r$ on the level of cohomology has kernels and cokernels killed by $N$ for a constant $N$ as in the theorem. Hence the same is true of the left map ${\alpha}_r$, as wanted.
\end{proof}

  The above theorem implies that  the logarithmic syntomic-\'etale cohomology is close to  the logarithmic syntomic-\'etale cohomology of the complement of the divisor at infinity. 
  \begin{corollary}
  \label{reduction}
  Let $X$ be a semistable scheme  over ${\so_K}$ with a divisor at infinity $D_{\infty}$. We treat it is as a log-scheme over $\so_K^{\times}$. Let $Y:=X\setminus D_{\infty}$ and let $ j_1: Y\hookrightarrow X$. 
 \begin{enumerate}
\item we have a natural quasi-isomorphism
$$
\tilde{\alpha}_r: \tau_{\leq r}\se_n(r)_X \stackrel{\sim}{\to} \tau_{\leq r}\R j_{1*}\se_n(r)_Y , \quad 0\leq r\leq p-2.
$$
\item there is a constant $N$ as in Theorem \ref{input1} and a natural morphism $$
{\alpha}_r: \se^{\prime}_n(r)_X\to \R j_{1*}\se^{\prime}_n(r)_Y  ,\quad r\geq 0,
$$
such that the induced map on cohomology sheaves in degree $q\leq r$ has kernel and cokernel annihilated by $p^{N}$. 
\end{enumerate}
  \end{corollary}
    \begin{proof}Note that $X_{\tr}=Y_K$ and set $j_2: Y_K\hookrightarrow Y$. We have $j=j_1j_2$. 
    By Theorem \ref{keylemma11}, both terms in the first claim are quasi-isomorphic to 
$$\tau_{\leq r}\R j_*\Z/p^n(r)_{X_{\tr}}=\tau_{\leq r}\R j_{1*}\tau_{\leq r}\R j_{2*}\Z/p^n(r)_{Y_K}.
$$
Hence they are quasi-isomorphic. The second claim of the corollary is proved in the same way.
  \end{proof}
\subsubsection{Nisnevich syntomic-\'etale cohomology}
    We will pass now to the Nisnevich topos of $X$. Denote by $\varepsilon: X_{\eet}\to X_{\Nis}$ the natural projection. For  $r\geq 0$, by applying $\R \varepsilon_*$ to the \'etale period map above and using that $\R \varepsilon_*i^*=i^*\R \varepsilon_*$\footnote{This equality fails for the projection to Zariski topology and is the reason we use Nisnevich topology instead of Zariski.} (c.f. \cite[2.2.b]{GD}),   we obtain     a natural map
$$
\tilde{\alpha}_{r}: \R \varepsilon_*\sss_n(r) \rightarrow i^*\R j_*\R \varepsilon_*{\mathbf Z}/p^n(r)'.
$$
 Composing with the map 
$\omega: \R \varepsilon_*\sss'_n(r)\to \R \varepsilon_*\sss_n(r)$ we get a natural, compatible with products,  
 morphism
$$
\alpha_{r}: \R \varepsilon_*\sss^{\prime}_n(r) \rightarrow i^*\R j_*\R \varepsilon_*{\mathbf Z}/p^n(r)'.
$$
Write, for simplicity,  $\sss_n(r)$ and $\sss^{\prime}_n(r)$  for the derived pushforwards of $\sss_n(r)$ and $\sss^{\prime}_n(r)$ from $X_{\eet}$ to $X_{\Nis}$. Same for $\se_n(r)$ and $\se^{\prime}_n(r)$.  
Notice that they are quasi-isomorphic to 
the complexes obtained by gluing the complexes of sheaves $\sss_n(r) $ and $\sss'_n(r) $ on $X_{0,\Nis}$ and the complexes of sheaves $\varepsilon_*j^{\prime}_*G{\mathbf Z}/p^n(r)^{\prime}$  on $X_{K,\Nis}$ by the maps $\tilde{\alpha}_{r}$ and $\alpha_r$. Hence we have the distinguished triangles
\begin{equation}
\label{Niis2}
j_{\Nis!} \R j_*^{\prime}\R \varepsilon_*{\mathbf Z}/p^n(r)^{\prime}\to \se_n(r)\to i_*\sss_n(r),\quad j_{\Nis !}  \R j_*^{\prime}{\mathbf Z}/p^n(r)^{\prime}\to \se^{\prime}_n(r)\to i_*\sss^{\prime}_n(r),
\end{equation}
as well as the natural maps
$$\tilde{\alpha}_r: \se_n(r)\to \R j_*\R \varepsilon_*{\mathbf Z}/p^n(r)^{\prime},\quad \alpha_r: \se_n(r)^{\prime}\to \R j_*\R \varepsilon_*{\mathbf Z}/p^n(r)^{\prime}
$$
compatible with the maps $\tilde{\alpha}_r$ and $\alpha_r$. For $a\geq 0$, we have the truncated version of the above - the distinguished triangles
\begin{equation}
\label{Niis3}
j_{\Nis !} \tau_{\leq a}\R j_*^{\prime}\R \varepsilon_*{\mathbf Z}/p^n(r)^{\prime}\to \tau_{\leq a}\se_n(r)\to i_*\tau_{\leq a}\sss_n(r),\quad j_{\Nis!} \tau_{\leq a} \R j_*^{\prime}\R \varepsilon_*{\mathbf Z}/p^n(r)^{\prime}\to \tau_{\leq a}\se^{\prime}_n(r)\to i_*\tau_{\leq a}\sss^{\prime}_n(r).
\end{equation}

Define the following complexes of sheaves on $X_{\Nis}$
\begin{align*}
\sss_n(r)_{\Nis}:=\tau_{\leq r}\sss_n(r), & \quad  \sss^{\prime}_n(r)_{\Nis}:=\tau_{\leq r}\sss^{\prime}_n(r);\\
\se_n(r)_{\Nis}:=\tau_{\leq r}\se_n(r), & \quad  \se^{\prime}_n(r)_{\Nis}:=\tau_{\leq r}\se^{\prime}_n(r).
\end{align*}
\begin{example}For $X=\Spec(W(k))$ we have 
\label{dim0}
$$
H^i(W(k),\sss_n(r)_{\Nis})=\begin{cases}
{\mathbf Z}/p^n & i=r=0,\\
W_n(k) & i=1, r\geq 1,\\
0 & \text{ otherwise}.
\end{cases}
$$
Moreover, the morphism $H^i(W(k),\se_n(r)_{\Nis})\to H^i(W(k),\sss_n(r)_{\Nis})$ is an isomorphism.

  To see the first claim, note that we have
$$
\sss_n(0)_{\eet}: W_n(k)\lomapr{1-\phi} W_n(k),\quad \sss_n(r)_{\eet}: 0\to W_n(k),\quad r\geq 1.
$$
It follows that
$$
\sss_n(0)_{\Nis}= {\mathbf Z}/p^n,\quad \sss_n(r)_{\eet}:=W_n(k)[-1],\quad r\geq 1.
$$

   For the second claim use the distinguished triangle (\ref{Niis3}) and the fact that $H^i(W(k),j_{\Nis !} \tau_{\leq a}\R j_*^{\prime}\R \varepsilon_*{\mathbf Z}/p^n(r))=0$,  $i\geq 0$, because
   $W(k)$ is henselian. 
\end{example}
\section{Syntomic cohomology  and motivic cohomology}
 Let $X$ be a smooth scheme  over $\so_K$. Let  ${\mathbf Z}(r)_{\M}$ denote  the complex of motivic sheaves
${\mathbf Z}(r)_{\M}:=X\mapsto z^r(X,2r-*) $ in the \'etale topology of $X$. Let ${\mathbf Z}/p^n(r)_{\M}:={\mathbf Z}(r)_{\M}\otimes\Z/p^n$.
Recall how the complex $z^r(X,*)$ is defined \cite{Bl}. Denote by $\triangle ^n$ the algebraic n-simplex 
$\Spec {\mathbf Z}[t_0,\ldots,t_n]/(\sum t_i-1)$. 
 Let $z^r(X,i)$ be the free abelian group generated by closed integral subschemas  of  codimension $r$ 
of $X\times \triangle ^i$ meeting all faces  properly.  Then $z^r(X,*)$ is the  chain complex thus defined with boundaries
 given by pullbacks of cycles along face maps. This complex is covariant for proper morphisms (with a shift in weight and degree) and contravariant for flat morphisms. 
 
   We know that in the Zariski topology 
$H^j(X_{\Zar},{\mathbf Z}/p^n(i)_{\M}) =H^j\Gamma(X_{\Zar},\Z/p^n(r)_{\M})$ is the 
Bloch higher Chow group \cite[Theorem 3.2]{GD} and  that this is also the case for the Nisnevich topology  \cite[Prop. 3.6]{GD}.  Locally, 
in the \'etale topology, when $p$ is invertible,  the \'etale cycle class map defines a quasi-isomorphism ${\mathbf Z}/p^n(r)_{\M}\simeq \Z/p^n(r) $; when $X$ is of characteristic $p$, then the logarithmic de Rham-Witt cycle class map defines a quasi-isomorphism $\Z/p^n(r)_{\M}\simeq W_n\Omega^r_{X,\log}[-r]$ \cite{GL}, where, for a log-scheme $Y$, $W_n\Omega^*_{Y,\log}$ denotes the sheaf of logarithmic de Rham-Witt differential forms \cite{Lor}.  
  Moreover, if $i:Z\hookrightarrow X$ is a closed subscheme of codimension $c$ with open complement $j:U\hookrightarrow X$ then the exact sequence
$$
0\to i_*{\mathbf Z}(r-c)_{M,Z}[-2c]\to {\mathbf Z}(r)_{M,X}\to j_*{\mathbf Z}(r)_{M,U}
$$
forms a distinguished triangle in the derived category of sheaves on $X_{*}$, $*$ denoting the Zariski or Nisnevich topology. We define motivic cohomology as
\begin{align*}
H^*_{\M}(X,\Z/p^n(r)) & :=H^*(X_{\Zar},\Z/p^n(r)_M)=H^*(X_{\Nis},\Z/p^n(r)_M);\\
H^*_{\M,\eet}(X,\Z/p^n(r)) & := H^*(X_{\eet},\Z/p^n(r)_M).
\end{align*}

For a  smooth scheme $Y$ over $\so_K$, we define its $p$-adic motivic cohomology as 
   $$  
   H^*_{\M}(Y,\Q_p(r))  := H^*(\holim_n\R\Gamma(Y_{\Zar},\Z/p^n(r)_{\M})\otimes\Q)=H^*(\holim_n\Gamma(Y_{\Zar},\Z/p^n(r)_{\M})\otimes\Q).   
   $$
   We define its \'etale version $H^*_{\M,\eet}(Y,\Q_p(r))$ in an analogous way.

  We list the following corollary of Theorem \ref{keylemma11}.
\begin{corollary}
Let $X$ be a smooth variety over $K$. Then there exists a natural syntomic cycle class map
$$
\cl^{\synt}_{i,r}: H^i_{\M}(X,\Q_p(r))\to H^i_{\synt}(X,\Q_p(r)),
$$
where the target group is the syntomic cohomology defined in \cite{NN}.
This map is compatible with the \'etale cycle class map, i.e., the following diagram commutes
$$
\xymatrix{
H^i_{\M}(X,\Q_p(r))\ar[d]^{\cl^{\synt}_{i,r}} \ar[dr]^{\cl^{\eet}_{i,r}}\\
H^i_{\synt}(X,\Q_p(r))\ar[r]^{\alpha^{NN}_{i,r}} &  H^i_{\eet}(X,\Q_p(r)),
}
$$
where $\alpha^{NN}_{i,r}$ is the period map defined in \cite{NN}.
Moreover, the cycle class map $\cl^{\synt}_{i,r}$  is an isomorphism for $i\leq r$. 
\end{corollary}
\begin{proof}
Consider the following diagram 
$$
\xymatrix{
\se^{\prime}_n(r)_{\Nis}\ar[r]^-{\alpha_r} & \tau_{\leq r}\R j_* \tau_{\leq r}\R \varepsilon_*\Z/p^n(r)^{\prime}\\
 & \R j_*\Z/p^n(r)_M^{\prime}\ar[u]^{\cl^{\eet}_{r}}_{\wr}\ar@{-->}[lu]^{\cl^{\synt}_{r}}
}
$$
The \'etale cycle class map $\cl^{\eet}_{r}$ is  a quasi-isomorphism by the  Beilinson -Lichtenbaum Conjecture (a corollary \cite{SV}, \cite{GL2} of the Bloch-Kato Conjecture proved by Voevodsky and Rost \cite{Wei}) and by \cite{GL} that give the quasi-isomorphism
$$
\Z/p^n(r)_M
\stackrel{\sim}{\to} \tau_{\leq r}\R \varepsilon_*\Z/p^n(r)
$$
and by the quasi-isomorphisms $j_*\Z/p^n(r)_M\stackrel{\sim}{\to}\R j_* \Z/p^n(r)_M$ and $\tau_{\leq r}\R j_*\Z/p^n(r)_M\stackrel{\sim}{\to} \R j_*\Z/p^n(r)_M$. 
Since, by Theorem  \ref{keylemma11}, the period map $\alpha_r$ is a $p^{N}$-quasi-isomorphism, we can define the syntomic cycle class map $\cl^{\synt}_{r}$
to make the above diagram commute. It induces the syntomic class map into syntomic cohomology 
$$
\cl^{\synt}_{r}: \R j_*\Z/p^n(r)_M\lomapr{\cl^{\synt}_{r}}  \se^{\prime}_n(r)_{\Nis}\to \sss^{\prime}_n(r)_{\Nis}\to \R\varepsilon_* \sss^{\prime}_n(r)_{\eet}
$$
By construction it is compatible with the \'etale cycle class map (via the map $\alpha_r$).

 Recall that the syntomic cohomology $H^i_{\synt}(X,\Q_p(r))$ is defined by $h$-sheafifying the (rational) Fontaine-Messing syntomic cohomology.
Everything being natural, the above construction of cycle classes $h$-sheafifies and gives the syntomic cycle class map
$$
\cl^{\synt}_{i,r}: H^i_{\M}(X,\Q_p(r))\to H^i_{\synt}(X,\Q_p(r)).
$$
For compatibility with the \'etale cycle class, it suffices  to check that $\alpha^{NN}_{i,r}=\alpha_{i,r}$ but this was done in \cite{New1}.

The last claim of the corollary follows from the fact that both $\alpha^{NN}_{i,r}$ and $\cl^{\eet}_{i,r}$ are isomorphisms for $i\leq r$ by \cite[Theorem A]{NN} and the Beilinson-Lichtenbaum Conjecture, respectively. 
\end{proof}
\subsection{Syntomic cohomology and logarithmic de Rham-Witt cohomology}
We will show in this section that adding logarithmic structure at the special fiber changes syntomic cohomology by logarithmic de Rham-Witt cohomology.
\begin{theorem}
\label{log-nonlog}
Let $X$ be a semistable scheme over $\so_K$ with a  smooth special fiber.  For a universal constant $N$, we have the following $p^{Nr}$-distinguished triangles of sheaves in the \'etale and Nisnevich topology of $X$, respectively.
\begin{align*}
 \sss^{\prime}_n(r)_{X} & \to \sss^{\prime}_n(r)_{X^{\times}}\to W_n\Omega^{r-1}_{X_0,\log}[-r],\\
 \sss^{\prime}_n(r)_{X,\Nis} & \to \sss^{\prime}_n(r)_{X^{\times},\Nis}\to W_n\Omega^{r-1}_{X_0,\log}[-r].
\end{align*}
Here we wrote $X^{\times}$ for the scheme $X$  with added  log-structure coming from the special fiber.
\end{theorem}
\begin{proof}
After setting up the local coordinates, we do, as an example, computations in dimension zero, where it becomes clear how to define the map to logarithmic de Rham-Witt differentials. Then we lift this computations to higher dimensions and globalize. 

{\em (1) Choice of local coordinates. }
   To construct the first distinguished triangle, we start with local computations. Let $d$ be a positive integer. Let  $R^0_K:=\so_K\{X_1^{\pm 1},\ldots, X_d^{\pm 1}\}$ be the $p$-adic completion of $\so_K[X_1^{\pm 1},\ldots, X_d^{\pm 1}]$. Let $R$ be the $p$-adic completion of an  \'etale algebra over $R^0_K$. Let $R^0_T$ be the $(p,T)$-adic completion of 
   $W(k)[T,X_1^{\pm 1},\ldots, X_d^{\pm 1}]$; take the map $R^0_T\mapsto R^0_K$, $T\mapsto \varpi$,  and take the (formally) \'etale lifting $R_T$ of $R$ to $R^0_T$.  Let $S_R$ be the $p$-adically complete PD-envelope of $R$ in $R_T$ equipped with the PD-filtration $F^rS_R$. We will write $S_K:=S_{\so_K}$. We have $S_R=R_T\widehat{\otimes} _{W(k)\{T\}}S_K$ with filtration $F^rS_R:=R_T\widehat{\otimes} _{W(k)\{T\}}F^rS_K$.   Let  $R^0:=W(k)\{X_1^{\pm 1},\ldots, X_d^{\pm 1}\}$ and let $R_{T,0}:=R_T/T$.

   We have the following diagram of maps (the right diagram is obtained by reducing the rings modulo $T$)
 \begin{equation}
 \label{diagram}
 \xymatrix{
  & \Spf S_R\ar[rd] & \\
 \Spf R\ar[d]\ar@{^{(}->}[ru]\ar@{^{(}->}[rr]\ar@{}[rr] & &\Spf R_T\ar[d]\\
 \Spf R^0_K\ar[d]\ar@{^{(}->}[rr] & & \Spf R^0_T\ar[d]\\
 \Spf \so_K \ar@{^{(}->}[rr]& & \Spf \so_F\{T\}
 }\quad \quad 
 \xymatrix{\\
 \Spec R_0\ar[d]\ar@{^{(}->}[r] & \Spf R_{T,0}\ar[d]\\
 \Spec R^0_{K,0}\ar[d]\ar@{^{(}->}[r] & \Spf R^{0}\ar[d]\\
 \Spec k  \ar@{^{(}->}[r] & \Spf \so_F
 }
 \end{equation}
Equip $R^0$ with Frobenius $\phi_{R^0}: X_i^{\pm1}\mapsto X_i^{\pm p}$. Equip $R^0_T$ with  Frobenius $\phi_{R^0_T}$ 
 compatible with $\phi_{S_K}$ ($T\mapsto T^p$) and with $\phi_{R^0}$, and equip $R_T$ with a Frobenius $\phi_{R_T}$ compatible with $\phi_{R^0_T}$. We will simply write $\phi$ for Frobenius if the domain of action is understood.  

    Set $\Omega_{S_R}:=S_R\otimes_{R_T}\Omega_{R_T}$. For $r\in \N$, we filter the de Rham complex $\Omega_{S_R}\kr$ by subcomplexes
    $$F^r\Omega_{S_R}\kr :=    F^rS_R\to F^{r-1}S_R\otimes_{R_T}\Omega_{R_T}\to F^{r-2}S_R\otimes_{R_T}\Omega^{2}_{R_T} \to\ldots   $$
 We define {\em the syntomic complex} of $R$  as
\begin{equation}
\label{kol1}
    S(R,r):=\Cone(F^r\Omega_{S_R}\kr\lomapr{p^r-\phi}\Omega_{S_R}\kr)[-1]
\end{equation}
 Set $\Omega_{ S^{\times}_R}:=S_R\otimes_{R_T}\Omega_{R^{\times}_T}$, where ${R^{\times}_T}$ is the ring $R_T$ with log-structure induced by $T$.
  We define the {\em log-syntomic complex} of $R$ as 
\begin{equation}
    \label{kol2}
    S(R^{\times},r):=\Cone(F^r\Omega_{S^{\times}_R}\kr\lomapr{p^r-\phi}\Omega_{S^{\times}_R}\kr)[-1].
\end{equation}

  For $n\in\N$, we define the syntomic and log-syntomic complexes modulo $p^n$ as  $S(R,r)_n:=S(R,r)\otimes_{\Z}\Z/p^n$, $S(R^{\times},r):=S(R^{\times},r)\otimes_{\Z}\Z/p^n$, respectively. 
In the case when $R$ is the $p$-adic completion of an \'etale algebra over $\so_K[X_1^{\pm 1},\cdots,X_d^{\pm 1}]$, we have 
\begin{align*}
S^{\prime}_n(r)_R   & =S(R,r)_n,\quad S^{\prime}_n(r)_{R^{\times}}=S(R^{\times},r)_n;\\
\holim_nS^{\prime}_n(r)_R & =S(R,r),\quad \holim_nS^{\prime}_n(r)_{R^{\times}}=S(R^{\times},r).
\end{align*}

 We would like  to separate the arithmetic and the geometric variables. Specifically,     we remove the differentials connected with the variable $T$ by setting $\Omega_{S^{\prime}_R}:=S_R\otimes_{R^0}\Omega_{R^0}$. Since  $\Omega_{W(k)[T]}=W(k)[T]dT$ we can dispose of this module of differentials by writing $df$ as $\partial f dT$ and   we can rewrite the above syntomic complex as the following homotopy limit
$$
S(R,r)=\left[\begin{aligned}\xymatrix@C=60pt{
F^r\Omega_{S^{\prime}_R}\kr \ar[d]^{\partial}\ar[r]^{p^r-p\kr\phi_{{\scriptscriptstyle\bullet}}} & \Omega_{S^{\prime}_R}\kr\ar[d]^{\partial}\\
F^{r-1}\Omega_{S^{\prime}_R}^{{\scriptscriptstyle\bullet}}\ar[r]^-{p^r-p^{{\scriptscriptstyle\bullet}+1}T^{p-1}\phi_{{\scriptscriptstyle\bullet}}} & \Omega_{S^{\prime}_R}^{{\scriptscriptstyle\bullet} }
}\end{aligned}\right]
$$
Here the map $\phi_{{\scriptscriptstyle\bullet}}:\Omega_{S^{\prime}_R}\kr\to \Omega_{S^{\prime}_R}\kr$ sends $\omega\in \Omega^k_{S^{\prime}_R}$ to $(\phi/p^k)(\omega)$.
  By adding logarithmic differentials $dT/T$ along the special fiber, we get  the following {\em log-syntomic complex}  
$$
S(R^{\times},r)=\left[\begin{aligned}\xymatrix@C=45pt{
 F^r\Omega_{S^{\prime}_R}\kr \ar[d]^{T\partial}\ar[r]^{p^r-p\kr\phi_{{\scriptscriptstyle\bullet}}} & \Omega_{S^{\prime}_R}\kr\ar[d]^{T\partial}\\
F^{r-1}\Omega_{S^{\prime}_R}^{{\scriptscriptstyle\bullet}}\ar[r]^-{p^r-p^{{\scriptscriptstyle\bullet}+1}\phi_{{\scriptscriptstyle\bullet}}} & \Omega_{S^{\prime}_R}^{{\scriptscriptstyle\bullet} }
}\end{aligned}\right]
$$
{\em (2) Dimension $0$.}
For $R=\so_K$, we obtain the following proposition.
\begin{proposition}
\label{ex1}
Let  $n\geq 1$. We have the following $p^{15}$-distinguished triangle of sheaves in the \'etale topology of $\Spec k$
$$
 S(\so_K,1)_n\to
S(\so_K^{\times},1)_n\to \Z/p^n[-1].
$$
For $r\neq 1$, the natural map $S(\so_K,r)_n\to
S(\so_K^{\times},r)_n$ is a $p^{15r}$-quasi-isomorphism.
\end{proposition}
\begin{proof}
We have the following two syntomic complexes.
\begin{align*}
S(\so_K,r):\quad & F^rS_K   \verylomapr{(\partial,p^r-\phi)}F^{r-1}S_K\oplus S_K\veryverylomapr{-(p^r-pT^{p-1}\phi)+\partial}S_K\\
S(\so_K^{\times},r):\quad  & F^rS_K   \verylomapr{(T\partial,p^r-\phi)}F^{r-1}S_K\oplus S_K\veryverylomapr{-(p^r-p\phi)+T\partial}S_K
\end{align*} 
  The residue map: $\Omega^1_{{\log},S_K^{[1]}}\to \so_F$ induces
the following   sequence of   complexes:
$$0\to S(\so_K^{[1]},r)\to S_{\log}(\so_K^{[1]},r)\to
\big[0\to\xymatrix{\so_F\ar[r]^-{-(p^{r}-p\varphi)}&\so_F}\big]\to 0,$$
where $R^{[1]}$ is "the ring of analytic functions over $F$ with integral values on the disk $v_p(T)\geq 1/e$" defined in \cite[Remark 2.2]{CN} and we define its syntomic complexes by analogous formulas to (\ref{kol1}) and (\ref{kol2}) replacing $S_R$ by $R^{[1]}$. We have the natural maps $S(\so_K,r)\to S(\so_K^{[1]},r)$ and $S(\so_K^{\times},r)\to S_{\log}(\so_K^{[1]},r)$ that are $p^{6r}$-quasi-isomorphisms \cite[Prop. 3.3]{CN}.
The above  sequence is  $p$-exact because $F^{s}S_K^{[1]}=\frac{E^s}{p^s}S_K^{[1]}$, for $E$ the minimal polynomial of $\varpi$ over $F$, 
which implies that  $F^s\Omega^1_{{\log},S_K^{[1]}}/F^s\Omega^1_{S_K^{[1]}}\cong
S_K^{[1]}/TS_K^{[1]}$ and  $S_K^{[1]}/TS_K^{[1]}=\so_F\oplus M$, where $M$ is $p$-torsion.

  Modulo $p^n$, we have $\so_{F,n}=W_n(k)$ and the exact sequence in the \'etale topology of $\Spec k$
$$
0\to \Z/p^n\to \so_{F,n}\lomapr{1-\phi}\so_{F,n}\to 0.
$$
For $r=0$, the map $\so_{F,n}\lomapr{p^r-p\phi}\so_{F,n}$ is an isomorphism since $1-p\phi$ is invertible. For $r >1$, the map $\so_{F,n}\lomapr{p^{r-1}-\phi}\so_{F,n}$ is an isomorphism as well since both $\phi$ and $p^{r-1}\phi^{-1}-1$ are invertible. This proves our proposition.
\end{proof}
 

{\em (3) Local computations in higher dimensions.}
  The computations in the above example generalize to any ring $R$. 
  \begin{lemma}
  There is a $p^{12r}$-distinguished triangle  in the \'etale topology of $\Spec R_0$
  $$S(R,r)_n\to S(R^{\times},r)_n\to W_n\Omega^{r-1}_{R_0,\log}[-r]
  $$ 
  \end{lemma}
  \begin{proof}
  First we  pass  from $S_R$ to $R^{[1]}$ (via a $p^{6r}$-quasi-isomorphism). Then we compute as in the proof of Proposition \ref{ex1} and obtain  the following  $p$-distinguished triangle 
$$ S(R^{[1]},r)\to S_{\log}(R^{[1]},r)\to
\big[\Omega_{R_{T,0}}^{{\scriptscriptstyle\bullet}}\veryverylomapr{p^r-p^{{\scriptscriptstyle\bullet}+1}\phi_{{\scriptscriptstyle\bullet}}}  \Omega_{R_{T,0}}^{{\scriptscriptstyle\bullet} }\big][-1].$$
 We note that  the complex $\Omega\kr_{R_{T,0}}$ computes the crystalline cohomology of $R_0$ over $W(k)$.

    Set $S:=R_{T,0}$. We claim that there exists a $p^{4r}$-quasi-isomorphism
    on the \'etale site of $\Spec R_0$ $$\big[\Omega_{S,n}^{{\scriptscriptstyle\bullet}}
    \veryverylomapr{p^r-p^{{\scriptscriptstyle\bullet}+1}\phi_{{\scriptscriptstyle\bullet}}}  \Omega_{S,n}^{{\scriptscriptstyle\bullet} }\big]\simeq W_n\Omega^{r-1}_{R_0,\log}[-r+1].$$
 Indeed, for $r=0$, the complex $\big[\Omega_{S,n}^{{\scriptscriptstyle\bullet}}
    \veryverylomapr{1-p^{{\scriptscriptstyle\bullet}+1}\phi_{{\scriptscriptstyle\bullet}}}  \Omega_{S,n}^{{\scriptscriptstyle\bullet} }\big]$  is acyclic because the map $ 1-p^{{\scriptscriptstyle\bullet}+1}$ is  invertible.
 Assume thus that $r\geq 1$ and take $s=r-1$. Set $$\hk(S,s)_n:=\big[\Omega_{S,n}^{{\scriptscriptstyle\bullet}}
    \veryverylomapr{p^s-p^{{\scriptscriptstyle\bullet}}\phi_{{\scriptscriptstyle\bullet}}}  \Omega_{S,n}^{{\scriptscriptstyle\bullet} }\big].
 $$
 This complex is $p^2$-quasi-isomorphic to the complex 
 $\big[\Omega_{S,n}^{{\scriptscriptstyle\bullet}}
    \veryverylomapr{p^r-p^{{\scriptscriptstyle\bullet}+1}\phi_{{\scriptscriptstyle\bullet}}}  \Omega_{S,n}^{{\scriptscriptstyle\bullet} }\big].$
 Using the global Frobenius lift on $S$ we get the following commutative diagram
 $$
 \xymatrix{\Omega_{S,n}^{{\scriptscriptstyle\bullet}}\ar[d]^{\Phi(\phi)}
    \ar[r]^-{p^s-p^{{\scriptscriptstyle\bullet}}\phi_{{\scriptscriptstyle\bullet}}} &   \Omega_{S,n}^{{\scriptscriptstyle\bullet} }\ar[d]^{\Phi(\phi)}\\
W_n\Omega\kr_{R_0}\ar[r]^-{p^s-p\kr\F} & W_n\Omega\kr_{R_0}/dV^{n-1}\Omega^{{\scriptscriptstyle\bullet}-1}_{R_0}
}
$$
We note here that the de Rham-Witt Frobenius $\F:W_{n+1}\Omega\kr_{R_0}\to W_n\Omega\kr_{R_0}$ and that $\F: \fil^nW_{n+1}\Omega\kr_{R_0}=V^n\Omega\kr_{R_0}+dV^{n}\Omega^{\scriptscriptstyle\bullet-1}_{R_0}\to dV^{n-1}\Omega^{\scriptscriptstyle\bullet-1}_{R_0}$. Hence $\F$ factorizes as in the above diagram. Moreover, since  $pdV^{n-1}\Omega^{*}_{R_0}=0$, we get the induced map  $p\F: W_n\Omega\kr_{R_0}\to W_n\Omega\kr_{R_0}$.

  The first vertical arrow in the above diagram is a quasi-isomorphism. The second one is a $p$-quasi-isomorphism since  $pdV^{n-1}\Omega^{*}_{R_0}=0$. Hence the complex $\hk(S,s)_n$ is $p^2$-quasi-isomorphic to the complex $[W_n\Omega\kr_{R_0}\verylomapr{p^s-p\kr\F}  W_n\Omega\kr_{R_0}/dV^{n-1}\Omega^{{\scriptscriptstyle\bullet}-1}_{R_0}]$.
We list the following properties of the latter complex.
\begin{enumerate}
\item For $t>s$, the map $W_n\Omega^t_{R_0}\verylomapr{1-p^{t-s}\F}W_n\Omega^t_{R_0}$ is an isomorphism (since $1-p^{t-s}\F$ is invertible).
\item For $t<s$, the map $$W_n\Omega^t_{R_0} \lomapr{p^{s-t}-\F}
W_n\Omega^t_{R_0}/dV^{n-1}\Omega^{t -1}_{R_0}$$ is a $p$-isomorphism. Indeed, for $p$-surjectivity, it suffices to note that $(p^{s-t}-\F)(V\alpha)=p^{s-t}V\alpha-p\alpha$, for $\alpha\in W_n\Omega^t_{R_0}/dV^{n-1}\Omega^{t-1}_{R_0}$, $t\leq s-1$. For $p$-injectivity, we note that if $(p^{s-t}-\F)(\alpha)=0$ for $\alpha\in W_n\Omega^t_{R_0}$ then $V(p^{s-t}-\F)(\alpha)=p^{s-t}V\alpha-p\alpha=0$. Hence $p^{s-t-1}V\alpha=\alpha$  which implies that $p^{n(s-t-1)}V^{n}\alpha=\alpha$. Hence $\alpha =0$.
\item The map $$ZW_n\Omega^{s-1}_{R_0} \lomapr{p-\F}
ZW_n\Omega^{s-1}_{R_0}/dV^{n-1}\Omega^{s-2}_{R_0}$$ is an $p^3$-isomorphism. For $p$-injectivity we use the point above. For $p^3$-surjectivity, we note that, for $\alpha\in ZW_n\Omega^{s-1}_{R_0} $ such that $d\alpha=0$ we have
$$
p\alpha=-(p-\F)\beta, \quad \beta=V\alpha+V^2\alpha+V^3\alpha\cdots\in VW_{n-1}\Omega^{s-1}_{R_0},
$$
and $p(1-\F)d\beta=0$. 
By \cite[Lemma 4.3]{BEK}, this implies that $pd\beta=0$.
\item There is an exact sequence 
$$
0\to W_n\Omega^s_{R_0,\log}\to W_n\Omega^s_{R_0}\lomapr{1-\F}W_n\Omega^s_{R_0}/dV^{n-1}\Omega^{s-1}_{R_0}\to 0
$$
 in the \'etale topology of $\Spec R_0$ \cite[Lemma 1.2]{CTSS}, \cite[Prop. 2.13]{Lor}. In the Nisnevich topology it is still exact on the left and in the middle.
\end{enumerate}
The above implies that there is a natural map $$W_n\Omega^s_{R_0,\log}[-s]\to [W_n\Omega\kr_{R_0}\verylomapr{p^s-p^{\scriptscriptstyle\bullet}\F}W_n\Omega\kr_{R_0}/dV^{n-1}\Omega^{\scriptscriptstyle\bullet-1}_{R_0}]$$ and that it is a $p^{4s}$-quasi-isomorphism in the \'etale topology of $\Spec R_0$, as wanted.
\end{proof}

{\em (4) Globalization.}
  The above local computations can be globalized in the following way. We note that we have actually proved above that we have the following $p^{4s}$-quasi-isomorphisms of sheaves on the \'etale site of $X_0$\footnote{The notation is slightly abusive here but we hope that this will not lead to confusion.}
$$
 W_n\Omega^s_{X_0,\log}[-s]\to [W_n\Omega\kr_{X_0}\lomapr{p^s-p\kr\F}W_n\Omega\kr_{X_0}/dV^n\Omega^{s-1}_{X_0}]\leftarrow{}
 [\sa_{\crr,n}\verylomapr{p^s-\phi}\sa_{\crr,n}],
$$
where $\sa_{\crr,n}$ is the sheaf $(U\to X_0)\mapsto \R\Gamma_{\crr}(U/W_n(k))$. 
The second quasi-isomorphism is \cite[Sec. II.1]{Il}. 
It suffices thus to construct a map
$$
\sss^{\prime}_n(r)_{X^{\times}}\to   [\sa_{\crr,n}\verylomapr{p^r-p
\phi}\sa_{\crr,n}][-1]
$$
and to show that the triangle 
$$
\sss^{\prime}_n(r)_X\to \sss^{\prime}_n(r)_{X^{\times}}\to   [\sa_{\crr,n}\verylomapr{p^r-p
\phi}\sa_{\crr,n}][-1]
$$
is $p^{6r}$-distinguished. 

   For that, consider the following two diagrams of compatible coordinate systems (localize on $X$ if necessary to get $X=\Spec A$).
\begin{equation*}
 \xymatrix{
  & \Spec D_{T,n}\ar[rd] & \\
 \Spec A_n\ar[d]\ar@{^{(}->}[ru]\ar@{^{(}->}[rr]\ar@{}[rr] & &\Spec B_{T,n}\ar[d]\\
 \Spec \so_{K,n}\ar[d]\ar@{^{(}->}[rr] & & \Spec \so_{F,n}[T]\ar[lld]\\
 \Spec \so_{F,n} 
 }\quad \quad 
 \xymatrix{  & \Spec D_{n}\ar[rd] & \\
 \Spec A_0\ar[d] \ar@{^{(}->}[ru]\ar@{^{(}->}[rr] & & \Spec B_{n}\ar[d]\\
 \Spec k\ar[d]\ar@{^{(}->}[rr] & & \Spec \so_{F,n}\ar@{=}[lld]\\
 \Spec \so_{F,n}
 }
 \end{equation*}
Here $B_{T,n}$ is smooth over $\so_{F,n}[T]$ and the hooked arrows are closed embeddings. We equip both rings with the log-structure associated to $T$. The right diagram is obtained by "reducing modulo $T$" the left diagram. It follows that the residue map $\Omega_{B^{\times}_{{T,n}}}\to\so_{B_n}$  induces a map $\Omega\kr_{D^{\times}_{{T,n}}}\to\Omega^{\scriptscriptstyle\bullet -1}_{D_n}$ (we note that the Frobenius $\phi$  on the domain is compatible with $p\phi$ on the target) and the sequence
\begin{equation}
\Omega\kr_{D_{{T,n}}}\to  \Omega\kr_{D^{\times}_{{T,n}}}\to\Omega^{\scriptscriptstyle\bullet -1}_{D_n}
\end{equation}
is exact. 
 These constructions glue in the usual way and we obtain a map $\sj_{X^{\times}_n}^{[r]}\to \sa_{\crr,n}[-1]$ and a sequence of complexes  of sheaves on the \'etale site of $X_0$
\begin{equation}
\label{global}
\sj^{[r]}_{X_n}\to \sj_{X^{\times}_n}^{[r]}\to \sa_{\crr,n}[-1],
\end{equation} 
where we wrote $\sj^{[r]}_{X^*_n}$ for the sheaf $(U\to X_n)\mapsto \R\Gamma_{\crr}(U, \sj^{[r]}_{X^*_n})$. 
Hence a sequence 
\begin{equation}
\label{global1}
 \sss^{\prime}_n(r)_X \to 
  \sss^{\prime}_n(r)_{X^{\times}} \to [\sa_{\crr,n}\lomapr{p^r-p\phi}\sa_{\crr,n}][-1]
\end{equation}
It 
is a $p^{6r}$-distinguished triangle:  this can be checked locally where we   can pass to the more convenient coordinate system from (\ref{diagram}) and use the computations we have done in the proof of Proposition \ref{ex1}. This concludes the construction of the first distinguished triangle of  our theorem.

    For the second triangle in the theorem, take the first triangle and push it down to the Nisnevich site. Since $\tau_{\leq 0}\R\epsilon_*W_n\Omega^{r-1}_{X_0,\log}\simeq W_n\Omega^{r-1}_{X_0,\log}$ \cite{K0}, it suffices to check that the map
$\sh^r(\sss^{\prime}_n(r)_{X^{\times},\Nis})\to W_n\Omega^{r-1}_{X_0,\log}$ is $p^{Nr}$-surjective, for a universal constant $N$. For that, Zariski localize and  consider the following commutative diagram
    $$
    \xymatrix{
    & & \Spf R_T\ar@/_17pt/[dl]_s\ar@/^17pt/[ddl]\\
    \Spec R_0\ar[d] \ar@^{{(}->}[r] & \Spf R_{T,0}\ar[d]\ar@^{{(}->}[ru] ^i\\
    \Spec k \ar@^{{(}->}[r] & \Spf \so_F
    }
    $$
  Here the map $s$ was chosen to commute with Frobenius and so that $si=\id$. We can use it to construct a section of the residue map
 $$
 S_{\log}(R^{[1]},r)\to
\big[\Omega_{R_{T,0}}^{{\scriptscriptstyle\bullet}}\veryverylomapr{p^r-p^{{\scriptscriptstyle\bullet}+1}\phi_{{\scriptscriptstyle\bullet}}}  \Omega_{R_{T,0}}^{{\scriptscriptstyle\bullet} }\big][-1].
 $$
 It follows that the map $H^i(S_{\log}(R^{[1]},r)_n)\to H^{i-1}(\big[\Omega_{R_{T,0}}^{{\scriptscriptstyle\bullet}}\veryverylomapr{p^r-p^{{\scriptscriptstyle\bullet}+1}\phi_{{\scriptscriptstyle\bullet}}}  \Omega_{R_{T,0}}^{{\scriptscriptstyle\bullet} }\big]_n)$ is surjective. Since, by the above computations, the Nisnevich sheaves 
 $\sh^{r-1}(\sa_{\crr,n}\lomapr{p^{r-1}-\phi}\sa_{\crr,n})$ and $W_n\Omega^{r-1}_{X_0,\log}$ are $p^{Nr}$-quasi-isomorphic, for a universal constant $N$, we are done. 
\end{proof}
\begin{corollary}
\label{needed0}
Let $X$ be a semistable scheme over $\so_K$ with a  smooth special fiber.  For a constant $N=N(p,e,r)$, we have the following natural $p^{N}$-quasi-isomorphism ($*$ denotes the   \'etale or the  Nisnevich topology of $X$)
$$
 \sss^{\prime}_n(r)_{X,*}\oplus W_n\Omega^{r-1}_{X_0,\log}[-r]\to \sss^{\prime}_n(r)_{X^{\times},*}
$$  
\end{corollary}
\begin{proof}It suffices to argue in the \'etale topology. Recall that 
in the proof of Theorem \ref{log-nonlog} we have obtained a $p^{6r}$-distinguished triangle
$$
 \sss^{\prime}_n(r)_X \to 
  \sss^{\prime}_n(r)_{X^{\times}} \lomapr{\res_T} [\sa_{\crr,n}\lomapr{p^r-p\phi}\sa_{\crr,n}][-1]
$$
and a $p^{4r}$ quasi-isomorphism 
\begin{equation}
\label{needed1}
 W_n\Omega^{r-1}_{X_0,\log}[-r+1]\simeq  [\sa_{\crr,n}\lomapr{p^r-p\phi}\sa_{\crr,n}].
\end{equation}
It suffices thus to construct a section of the residue map. For $r=0$ there is nothing to prove so we will assume that $r >0$. 

   For that, consider the following commutative diagram
$$
\xymatrix{
& \R\Gamma_{\crr}(X_1/W_n(k))\ar[d]^{i^*}\ar[dl]_{\wedge \dlog T}\\
\R\Gamma_{\crr}(X^{\times}_1/W_n(k))[1]\ar[r]^-{\res_T} & \R\Gamma_{\crr}(X_0/W_n(k)),
}
$$
where $i:X_0\hookrightarrow X_1$ is the natural closed immersion.
Since the map $i^*: \R\Gamma_{\crr}(X_1/W_n(k))^{\phi=p^{r-1}}\to 
 \R\Gamma_{\crr}(X_0/W_n(k))^{\phi=p^{r-1}}$ is a $p^{N(p,e)}$ quasi-isomorphism \cite[Remark 5.9]{CN}, the maps in the above diagram induce a $p^{N(p,e)}$-map
$ \R\Gamma_{\crr}(X_0/W_n(k))^{\phi=p^{r-1}}\to \sss^{\prime}_n(r)_{X^{\times}}[1]$
that is a $p^{N(p,e,r)}$-section of the residue map  on the scheme $X_0$. It is natural on the \'etale site of $X_0$ (it depends only on the uniformizer  $\varpi$) hence gives a section on the level of sheaves that we wanted. 
\end{proof}

  Let, for $*$ denoting the \'etale or the Nisnevich topology,
\begin{align*}
  \R\Gamma(X_{*},\sss(r))_{\Q}   & :=\holim_n\R\Gamma(X_{*},\sss_n(r))\otimes\Q   \stackrel{\sim}{\to} \holim_n\R\Gamma(X_{*},\sss^{\prime}_n(r))\otimes\Q,\\
  \R\Gamma(X_{*},\se(r))_{\Q}   & :=\holim_n\R\Gamma(X_{*},\se_n(r))\otimes\Q   \stackrel{\sim}{\to} \holim_n\R\Gamma(X_{*},\se^{\prime}_n(r))\otimes\Q,\\   
  \R\Gamma(X_{*},W\Omega^{r-1}_{X_0,\log})_{\Q}  & := \holim_n\R\Gamma(X_*,i_*W_n\Omega^{r-1}_{X_0,\log}).
\end{align*}
By (\ref{needed1}), we have $$
\R\Gamma(X_{*},W\Omega^{r-1}_{X_0,\log})_{\Q}\simeq 
\R\Gamma_{\crr}(X_0/F)_{\Q}^{\phi=p^{r-1}}[-r+1],
$$ 
where, for a scheme $Y$ over $W(k)$,  we set $\R\Gamma_{\crr}(Y/F):= \R\Gamma_{\crr}(Y/W(k))_{\Q}:=\holim_n\R\Gamma_{\crr}(Y_1/W_n(k))\otimes {\Q}$.
The following corollary follows immediately from Corollary \ref{needed0}.
\begin{corollary}
\label{mot2}
Let $X$ be a semistable scheme over $\so_K$ with a  smooth special fiber.  We have the following natural quasi-isomorphisms
$$
 \R\Gamma(X_{*},\sss(r))_{\Q}\oplus \R\Gamma(X_{*},W\Omega^{r-1}_{X_0,\log})_{\Q}[-r]  \stackrel{\sim}{\to} \R\Gamma(X^{\times}_{*},\sss(r))_{\Q}.
 $$
\end{corollary}
 \begin{corollary}
\label{mot1}
Let $X$ be a semistable scheme  over $\so_K$ with a smooth special fiber. For a  constant $N$ as in Theorem \ref{input1}, we have the following $p^{N}$-distinguished triangle of sheaves in the \'etale topology of $X_0$
\begin{equation}
\label{seqq}
 \sss^{\prime}_n(r)_X\to \tau_{\leq r}i^*\R j_*\Z/p^n(r)^{\prime}\to W_n\Omega^{r-1}_{X_0,\log}[-r].
\end{equation}
Moreover, for a constant $N=N(p,e,r)$ we have the following $p^N$-quasi-isomorphism
$$
\sss^{\prime}_n(r)_X\oplus W_n\Omega^{r-1}_{X_0,\log}[-r]\to\tau_{\leq r}i^*\R j_*\Z/p^n(r)^{\prime}.
$$
\end{corollary}
\begin{proof}
This immediately follows from Theorem \ref{log-nonlog}, Theorem \ref{input1}, and Corollary \ref{needed0}.
\end{proof}
\begin{remark}For $r\leq p-2$, the distinguished triangle (\ref{seqq}) was constructed before by  
Kurihara. No additional constants are needed in this case. 
\begin{theorem}(\cite[1]{Kur})
Let $X$ be a smooth  scheme over $\so_K$. For $r\leq p-2$, we have the following distinguished triangle of sheaves in the \'etale topology of $X_0$
$$
\sss_n(r)_X\to \tau_{\leq r}i^*\R j_*\Z/p^n(r)\to W_n\Omega^{r-1}_{X_0,\log}[-r].
$$
\end{theorem} 
It is easy to see that the above theorem holds also for schemes $X$ that are semistable over  $\so_K$ with a smooth special fiber, i.e., that we have the following distinguished triangle
$$
\sss_n(r)_X\to \tau_{\leq r}i^*\R j_*\Z/p^n(r)\to W_n\Omega^{r-1}_{X_0,\log}[-r],\quad r\leq p-2.
$$
Indeed, it suffices to note that all the terms involved have Gysin sequences \cite{Ts1} and to use the above theorem. In particular, in view of Theorem \ref{input0}, we have the following distinguished triangle
$$
\sss_n(r)_X\to \sss_n(r)_{X^{\times}}\to W_n\Omega^{r-1}_{X_0,\log}[-r],\quad r\leq p-2,
$$
a "small twists" analog of the distinguished triangles from  Theorem \ref{log-nonlog}.
\end{remark}
\subsection{Syntomic cohomology and motivic cohomology}
The main theorem of this section shows that, in  \'etale topology,  syntomic-\'etale complexes on smooth schemes over $\so_K$ approximate motivic complexes. 
\begin{theorem}
\label{keylemma10}
Let $X$ be a semistable scheme  over $\so_K$ with a smooth special fiber.  
Let $j^{\prime}: X_{\tr}\hookrightarrow X$ be the natural open immersion. Then
\begin{enumerate}
\item there is a natural cycle class map
$$
\cl^{\synt}_r: \R j^{\prime}_*{\mathbf Z}/p^n(r)_{\M}\to \se_n(r)_{\Nis},\quad 0\leq r\leq p-2.
$$
It is a quasi-isomorphism.
\item there is a   natural cycle class map 
$$\cl^{\synt}_r:\R j^{\prime}_*{\mathbf Z}/p^n(r)_{\M}\to \se^{\prime}_n(r)_{\Nis},\quad  r\geq 0.
$$
It is a  $p^N$-quasi-isomorphism for  a constant $N$ as in Theorem \ref{input1}.
\end{enumerate}
We have analogous statements in the \'etale topology. These cycle class maps are compatible (via the localization map and the period map) with the \'etale cycle class maps.
\end{theorem}
\begin{proof}We start with the Nisnevich topology. We will prove the second claim, the proof of the first one being analogous.
Consider the following commutative diagram
$$
\xymatrix{
j_{\Nis !} \tau_{\leq r}\R j_{K,*}^{\prime}\R\varepsilon_*{\mathbf Z}/p^n(r)^{\prime}\ar[r]\ar[d]^{\wr} & \se^{\prime}_n(r)_{X,\Nis}\ar[r]\ar[d] &  i_*\sss^{\prime}_n(r)_{X,\Nis}\ar[d]\\
j_{\Nis !} \tau_{\leq r}\R j_{K,*}^{\prime}\R\varepsilon_*{\mathbf Z}/p^n(r)^{\prime}\ar[r] &  \se^{\prime}_n(r)_{X^{\times},\Nis}\ar[r] &  i_*\sss^{\prime}_n(r)_{X^{\times},\Nis}\ar[d]\\
& & i_*W_n\Omega^{r-1}_{X_0,\log}[-r]
}
$$
The two rows are distinguished triangles; the right  column is a $p^{Nr}$-distinguished triangle, for a universal constant $N$, by Theorem \ref{log-nonlog}.
It follows that we have the $p^{Nr}$-distinguished triangle
\begin{equation}
\label{split}
 \se^{\prime}_n(r)_{X,\Nis}\to \se^{\prime}_n(r)_{X^{\times},\Nis}\to i_*W_n\Omega^{r-1}_{X_0,\log}[-r].
\end{equation}
Let $Y=X_{\tr}$. By functoriality we get the following map of $p^{Nr}$-distinguished triangles
$$
\xymatrix{\se^{\prime}_n(r)_{X,\Nis}\ar[d] \ar[r] & \se^{\prime}_n(r)_{X^{\times},\Nis}\ar[d]^{\wr}\ar[r] &i_*W_n\Omega^{r-1}_{X_0,\log}[-r]\ar[d]^{\wr}\\
\R j_*^{\prime}\se^{\prime}_n(r)_{Y,\Nis}\ar[r] & \R j_*^{\prime}\se^{\prime}_n(r)_{Y^{\times},\Nis}\ar[r] &\R j_*^{\prime}i_*W_n\Omega^{r-1}_{Y_0,\log}[-r]}
$$
The right vertical arrow is a quasi-isomorphism since $M_{X_0}=j^{\prime}_*\so_{X_0,\tr}^{*}$.
The middle vertical arrow is a $p^{Nr}$-quasi-isomorphism by Corollary \ref{reduction}.
Hence the left vertical  arrow is a $p^{Nr}$-quasi-isomorphism
and we may assume that the horizontal divisor of $X$ is trivial. 

  Consider the following diagram
\begin{equation}
\label{motivic}
\xymatrix{
\se^{\prime}_n(r)_{X,\Nis}\ar@{.>}[d]\ar[r] &  \se^{\prime}_n(r)_{X^{\times},\Nis} \ar[d]^{{\alpha}_r}_{\wr}\ar[r] & i_*W_n\Omega^{r-1}_{X_0,\log}[-r]\ar@{=}[d]\\
\scc_n(r) \ar[r] & \tau_{\leq r}\R j_* \R\varepsilon_*\Z/p^n(r)^{\prime}_{X_K} \ar[r]^-{\kappa} & i_*W_n\Omega^{r-1}_{X_0,\log}[-r]}
\end{equation}
Here the map $\kappa$ is induced from a map $\tau_{\leq r}i^*\R j_*\Z/p^n(r)\to W_n\Omega^{r-1}_{X_0,\log}[-r]$ of sheaves on the \'etale site of $X_0$ defined as the composition of the canonical map $\tau_{\leq r}i^*\R j_*\Z/p^n(r)\to i^*\R^r j_*\Z/p^n(r)[-r]$ and the symbol map $i^*\R^r j_*\Z/p^n(r)\to W_n\Omega^{r-1}_{X_0,\log}$. The latter is defined by observing that  $i^*\R^r j_*\Z/p^n(r)$ is locally generated by symbols $\{f_1,\ldots,f_r\}$ for $f_i\in i^*j_*\so_{X_K}^{*}$ \cite[Cor. 6.1.1]{BK}. By multilinearity, each symbol can be written as a sum of symbols of the form $\{f_1,\ldots,f_r\}$ and $\{f_1,\ldots,f_{r-1},\varpi\}$ for $f_i\in i^*\so^{*}_X$. Then $\kappa$ sends the former to zero and the latter to $\dlog [\overline{f}_1]\wedge\cdots\wedge \dlog [\overline{f}_{r-1}]$ where $\overline{f}_i$ is the reduction of $f_i$ to $\so_{X_0}^{*}$.
We defined  $\scc_n(r)$  as the mapping fiber of the map $\kappa$.  

  We claim that the right square of the diagram $p^{N}$-commutes for a constant as in the statement of the theorem. Indeed, we note that we can pass to the \'etale site and there it suffices to show that the following diagram of maps of sheaves  $p^{N}$-commutes
$$
\xymatrix{
\sh^r(\sss^{\prime}_n(r)_{X^{\times}})\ar[r]^{\beta}\ar[d]^{\alpha_r} & W_n\Omega^{r-1}_{X_0,\log}\\
i^*\R^rj_*\Z/p^n(r)^{\prime}_{X_K}\ar[ru]^{\kappa}
}
$$
Since the map $\alpha_r$ is a $p^{N}$-isomorphism and the sheaf $i^*\R ^rj_*\Z/p^n(r)_{X_K}$ is generated locally by symbols it suffices to check that the map $\beta$ sends the symbol $\{ f_1,\ldots, f_{r}\}$, $f_i\in i^*\so_{X}^{*}$, to zero and the symbol $\{f_1,\ldots,f_{r-1},\varpi\}$, $f_i\in i^*\so_{X}^{*}$, to $\dlog [\overline{f}_1]\wedge\cdots\dlog [\overline{f}_{r-1}]$. But this follows easily from the definition of the symbol maps (\ref{symbols}).

  It follows that the left vertical map in the diagram \ref{motivic} exists. It is unique because $$\Hom(\se^{\prime}_n(r)_{X,\Nis},W_n\Omega^{r-1}_{X_0,\log}[-r-1])=0$$ for degree reasons. It is clearly a quasi-isomorphism. All of the above has to be taken in the $p^{N}$-sense. 

 It remains now to show that there exists a $p^{Nr}$-quasi-isomorphism $\Z/p^n(r)_{M}\to \scc_n(r)$, for a universal constant $N$. We proceed as in \cite[p. 14]{GD}. 
 Consider the following diagram of distinguished triangles (the complex $\scc_n^{\prime}(r)$ is defined by the bottom triangle and $p^{Nr}$-quasi-isomorphic to the complex $\scc_n(r)$)
 \begin{equation}
 \label{nis}
\xymatrix{
\Z/p^n(r)_{M,X} \ar@{.>}[d]\ar[r] & j_*\Z/p^n(r)_{M,X_{K}} \ar[r]\ar[d]^{\wr}& i_*\Z/p^n(r-1)_{M,X_{0}}[-1]\ar[d]^{\wr}\\
\scc_n^{\prime}(r)\ar[r] &  \tau_{\leq r}\R j_*\R \varepsilon_*\Z/p^n(r)_{X_{K}} \ar[r]^{\kappa} & i_*W_n\Omega^{r-1}_{X_0,\log}[-r]
}
\end{equation}
The middle and the right vertical maps are induced by the \'etale and  the logarithmic de Rham-Witt cycle class map, respectively. They are quasi-isomorphisms by the  Beilinson -Lichtenbaum Conjecture. The right square commutes: pass to the \'etale site and there this fact  was shown in   \cite[p. 14]{GD}. Hence the left vertical map exists, is unique, and a quasi-isomorphism as well.  This concludes the proof of our theorem.

  For the \'etale topology, the computations are analogous but the diagram (\ref{nis}) has to be replaced with the following one
  \begin{equation*}
\xymatrix{
\Z/p^n(r)_{M,X} \ar@{.>}[d]\ar[r] & \tau_{\leq r}\R j_*\Z/p^n(r)_{M,X_{K}} \ar[r]\ar[d]^{\wr}& \tau_{\leq r}(i_*\R i^!\Z/p^n(r)_{M,X}[1])\ar[d]^{\wr}\\
\scc_n^{\prime}(r)\ar[r] &  \tau_{\leq r}\R j_*\Z/p^n(r)_{X_{K}} \ar[r]^{\kappa} & i_*W_n\Omega^{r-1}_{X_0,\log}[-r]
}
\end{equation*}
The right vertical arrow is a quasi-isomorphism by \cite[p. 14]{GD}.
\end{proof}

    We list several, more or less  immediate, corollaries of the above theorems (we set   $\alpha:=\eet,\Nis$).\begin{corollary}Let $X$ be a smooth scheme  over $\so_K$.  
  We have
\begin{enumerate}
\item $H^*_{\alpha}(X,\se_n(r))\simeq H^*_{\M,\alpha}(X,\Z/p^n(r)),\quad r\leq p-2$;
\item the kernel and the cokernel of the cycle class map
$$ H^{*}_{\M,\alpha}(X,\Z/p^n(r))\to H^*_{\alpha}(X,\se^{\prime}_n(r))
$$ 
are annihilated by $p^N$, where   $N$ denotes the constant from  Theorem \ref{input1}. Hence 
$$H^*_{\alpha}(X,\se(r))_{\Q}\simeq H^*_{\M,\alpha}(X,\Q_p(r)).$$ 
\end{enumerate}
\end{corollary}

  In a more familiar language of syntomic cohomology, the above theorem and corollary  can be stated in the following way. 
\begin{corollary}
Let $X$ be a semistable scheme  over $\so_K$ with a smooth special fiber. Let $j^{\prime}: X_{\tr}\hookrightarrow X$ be the natural open immersion. Then, on the \'etale site of $X_0$, 
\begin{enumerate}
\item there is a natural quasi-isomorphism \cite{GD}
$$
\sss_n(r)\simeq i^*\R j^{\prime}_*{\mathbf Z}/p^n(r)_{\M},\quad 0\leq r\leq p-2.
$$
\item there is a constant $N$ as in Theorem \ref{input1} and a   natural $p^N$-quasi-isomorphism
$$\sss^{\prime}_n(r)\simeq i^*\R j^{\prime}_*{\mathbf Z}/p^n(r)^{\prime}_{\M},\quad r\geq 0,
$$
\end{enumerate}
\end{corollary}
\begin{corollary}
\label{cor1}
Let $X$ be a proper semistable  scheme  over $\so_K$ with a smooth special fiber.  
  We have
\begin{enumerate}
\item $H^*_{\alpha}(X,\sss_n(r))\simeq H^*_{\M,\alpha}(X_{\tr},\Z/p^n(r)),\quad r\leq p-2$;
\item the kernel and the cokernel of the cycle class map
$$ H^{*}_{\M,\alpha}(X_{\tr},\Z/p^n(r))\to H^*_{\alpha}(X,\sss^{\prime}_n(r))
$$ 
are annihilated by $p^N$, where $N$ denotes the constant from  Theorem \ref{input1}. Hence 
$$H^*_{\alpha}(X,\sss(r))_{\Q}\simeq H^*_{\M,\alpha}(X_{\tr},\Q_p(r)).$$ 
\end{enumerate}
\end{corollary}
\begin{corollary}
Let $X$ be a proper semistable  scheme  over $\so_K$ with a smooth special fiber. Then the claims of Corollary \ref{cor1} hold for $X_{\so_{\ovk}}$ (in place of $X$)\footnote{Syntomic cohomology of $X_{\so_{\ovk}}$ is defined in the same way as the one of $X$.}. Moreover, for $i\leq r$, we have the following commutative diagram
$$
\xymatrix{H^i_{\M}(X_{\so_{\ovk},\tr},\Q_p(r))\ar[r]^{j^*}_{\sim} \ar[d]^{\cl^{\synt}_{i,r}}_{\wr}& H^i_{\M}(X_{\ovk,\tr},\Q_p(r))\ar[d]^{\cl^{\eet}_{i,r}}_{\wr}\\
H^i(X_{\so_{\ovk}},\sss(r))_{\Q}\ar[r]^{\alpha_{i,r}} & H^i_{\eet}(X_{\ovk,\tr},\Q_p(r))
}.
$$ \end{corollary}
\begin{proof}
The first claim follows from Corollary \ref{cor1} by passing to limit over finite extensions of $K$ in $\ovk$. 
The fact that the localization map $j^*$ is an isomorphism was proved in \cite[Lemma 3.1]{N2}. 
\end{proof}
\begin{remark}
For $X$ proper the above diagram was studied in \cite{N2} (see \cite{ICM} for a brief survey): it was constructed first for the Chern classes from $p$-adic $K$-theory and then for motivic cohomology by studying compatibility of Chern classes with operations on $K$-theory. This did not use the Fontaine-Messing period map $\alpha_{i,r}$ but instead
a period map $\alpha_{i,r}: H^i_{\eet}(X_{\ovk,\tr},\Q_p(r))\to H^i(X_{\so_{\ovk}},\sss(r))_{\Q}$ was defined using the above diagram. The fact that it is an isomorphism followed from the proof of the Crystalline Conjecture and implied that so is the syntomic cycle class map $\cl^{\synt}_{i,r}$.

  For an open $X$ as above, the situation is, at the moment, reversed. We defined log-syntomic $p$-adic Chern classes \cite{New} using the (universal) syntomic cycle class maps constructed in this paper. 
\end{remark}

\appendix
\section{Comparison of crystalline, convergent, and rigid syntomic cohomologies}
We will compare  crystalline, convergent, and rigid syntomic cohomologies for smooth schemes over $\so_K$ with normal crossing compactifications. 
Let $X$ be a smooth scheme over $\so_K$. Recall Besser's definition of rigid syntomic cohomology \cite{Be1}
$$
\R\Gamma^{\rig}_{\synt}(X,r):=[\R\Gamma_{\rig}(X_0/F)\oplus F^r\R\Gamma_{\dr}(X_K)\lomapr{f}
\R\Gamma_{\rig}(X_0/F)\oplus \R\Gamma_{\rig}(X_0/K)],\quad r\geq 0.
$$
Here $\R\Gamma_{\rig}(\cdot)$ denotes the rigid cohomology and 
$f: (x,y)\mapsto ((p^r-\phi)(x), \spb(y)-x)$, where $\spb$  is the Berthelot's specialization map. 
\begin{proposition}
\label{a1}
Let  $X$ be a proper semistable scheme over $\so_K$ with a smooth special fiber. There is a natural quasi-isomorphism
$$
\R\Gamma^{\rig}_{\synt}(X_{\tr},r)\simeq 
\R\Gamma_{\synt}(X,r),\quad r\geq 0.
$$
\end{proposition}
\begin{proof}As usual we consider $X$ as a log-scheme.
We can write 
$$
\R\Gamma_{\synt}^{\rig}(X_{\tr},r)\simeq [\R\Gamma_{\rig}(X_{0,\tr}/F)^{\phi=p^r}\to
\R\Gamma_{\rig}(X_{0,\tr})/F^r\R\Gamma_{\dr}(X_{K,\tr})]
$$
Since we have
$$
\R\Gamma_{\synt}(X,r)\simeq[\R\Gamma_{\crr}(X/F)^{\phi=p^r}\to \R\Gamma_{\dr}(X_K)/F^r],
$$
it suffices to construct a map
$$
\R\Gamma_{\crr}(X/F) \to \R\Gamma_{\rig}(X_{0,\tr}/F)
$$
that is compatible (in the dg category sense) with Frobenius and the specialization map from de Rham cohomology.  This is accomplished by the following commutative diagram.
 
 $$
 \xymatrix{&  \R\Gamma_{\crr}(X_{1}/F) \ar[dl]^{i^*}\ar[r] &  \R\Gamma_{\crr}(X_{1}/K) & \R\Gamma_{\dr}(X_K)\ar[dr]^{\sim} \ar[l]_{\sigma_{\crr}}^{\sim}
 \ar@/^17pt/[ddl]_-{\sigma_{\conv}}^{\sim}\\
 \R\Gamma_{\crr}(X_{0}/F) & \R\Gamma_{\conv}(X_{1}/F) \ar[d]^{\wr}_{i^*}\ar[r]\ar[u]_{\alpha_{1,F}} &  \R\Gamma_{\conv}(X_{1}/K)\ar[d]^{\wr}_{i^*} \ar[u]_{\alpha_{1,K}}& &\R\Gamma_{\dr}(X_{K,\tr})\ar@/^22pt/[ddll]^{\spb}\\
   & \R\Gamma_{\conv}(X_{0}/F)\ar[ul]^{\alpha_0}_{\sim} \ar[d]^{\wr}\ar[r] &  \R\Gamma_{\conv}(X_{0}/K)\ar[d]^{\wr} \\
   & \R\Gamma_{\rig}(X_{0,\tr}/F) \ar[r] &  \R\Gamma_{\rig}(X_{0,\tr}/K) 
 }
 $$ Here $ \R\Gamma_{\conv}(\cdot)$ denotes the (logarithmic) convergent cohomology \cite{Og, BOg, Sh1} that is used classically to connect rigid cohomology with crystalline cohomology. 
 The quasi-isomorphisms between the rigid and the convergent cohomology at the bottom of the diagram are proved in \cite[Cor. 2.4.13]{Sh1}. The maps $i^*$ are quasi-isomorphisms by invariance of convergent cohomology under nilpotent thickenings \cite[1.14.3]{BOg}. The map $\alpha_0$ is a quasi-isomorphism by \cite[Theorem 3.1.1]{Sh1}. The top map $i^*$ is a quasi-isomorphism on $\phi$-eigenspaces \cite[Remark 5.9]{CN}; hence so is the map $\alpha_{1,F}$. The quasi-isomorphisms $\sigma_{\crr}, \sigma_{\conv}$ are simply the crystalline and the convergent \cite[2.3]{Sh1} Poincar\'e Lemmas, respectively. It follows that the specialization map $\spb$ as well as the map $\alpha_{1,K}$ are quasi-isomorphisms as well.
 \end{proof}
 \begin{remark}
 Recall that Besser's definition of rigid syntomic cohomology is modeled on the definition of convergent syntomic cohomology \cite{N1}. In its logarithmic form the latter is defined as the following mapping fiber
 $$
 \R\Gamma_{\synt}^{\conv}(X,r):= [\R\Gamma_{\conv}(X_{0}/F)^{\phi=p^r}\to
\R\Gamma_{\conv}(X_{0}/K)/F^r\R\Gamma_{\conv}(X_0/K)]
 $$
 The proof of the above proposition shows that,  for a proper and semistable scheme over $\so_K$ with a smooth special fiber, we have natural quasi-isomorphisms
 \begin{equation}
 \label{kwak}
 \R\Gamma^{\rig}_{\synt}(X_{\tr},r)\simeq \R\Gamma^{\conv}_{\synt}(X,r)\simeq \R\Gamma_{\synt}(X,r),\quad r\geq 0.
 \end{equation}
 In the proper case this was shown in \cite[Prop. 9.8]{Be1}.
 \end{remark}
   For a variety $Y$ over $K$, let $\R\Gamma^{NN}_{\synt}(Y,r)$ denote the syntomic cohomology defined in \cite{NN}.
 \begin{corollary}Let  $X$ be a proper semistable scheme over $\so_K$ with a smooth special fiber. 
 There is a natural distinguished triangle
 $$
 \R\Gamma^{\rig}_{\synt}(X_{\tr},r) \oplus \R\Gamma(W\Omega^{r-1}_{X_0,\log})_{\Q}[-r]\stackrel{\sim}{\to} \R\Gamma^{NN}_{\synt}(X_{K,\tr},r).
 $$
 \end{corollary}
 \begin{proof}
 Since we have a canonical quasi-isomorphism \cite[Prop. 3.18]{NN}
 $$
 \R\Gamma_{\synt}(X^{\times},r)_{\Q}\stackrel{\sim}{\to}\R\Gamma^{NN}_{\synt}(X_{K,\tr},r),$$
 this follows immediately from Proposition \ref{a1} and Corollary \ref{mot2}. 
 \end{proof}


\begin{thebibliography}{Fa0}
\bibitem{BOg}  P.~Berthelot, A.~Ogus, {\em $ F$-isocrystals and de Rham cohomology. I.} Invent. Math. 72 (1983), no. 2, 159--199. 
\bibitem{Be1}  A.~Besser, {\em Syntomic regulators and $p$-adic integration. I. Rigid syntomic regulators.} Proceedings of the Conference on $p$-adic Aspects of the Theory of Automorphic Representations (Jerusalem, 1998). Israel J. Math. 120 (2000), part B, 291--334. 
\bibitem{Bl} S.~Bloch, {\em  Algebraic cycles and higher $K$-theory}, Adv. in Math. 61 (1986), no. 3, 267--304. 
\bibitem{BEK} S.~Bloch, H.~Esnault, M.~Kerz, {\em $p$-adic deformation of algebraic cycle classes},  Invent. Math. 195 (2014), 673--722.
\bibitem{BK} S.~Bloch, K.~Kato, {\em $p$-adic \'etale cohomology}. Inst. Hautes \'Etudes Sci. Publ. Math. No. 63 (1986), 107-152. 
\bibitem{CTSS} J.-L.~Colliot-Th\'el\`ene, J.-J.~Sansuc, C.~Soul\'e, {\em Torsion dans le groupe de Chow de
codimension deux}, Duke Math. J. 50 (1983) no. 3, 763--801.
\bibitem{BM} Ch.~Breuil, W.~Messing, {\em Torsion \'etale and crystalline cohomologies}, Cohomologies $p$-adiques et applications arithm\'etiques, II. Ast\'erisque No. 279 (2002), 81--124.  
 \bibitem{CN} P.~Colmez, W.~Nizio{\l}, {\em Syntomic complexes and $p$-adic nearby cycles}, arXiv:1505.06471. 
\bibitem{FM} J.-M.~Fontaine and  W.~Messing, {\em  $p$-adic periods and $p$-adic \'{e}tale
cohomology}, Current Trends in Arithmetical Algebraic Geometry
(K.~Ribet, ed.), Contemporary
Math., vol. 67, Amer. Math. Soc., Providence, 1987, 179--207.
\bibitem{GD} T.~Geisser, {\em Motivic cohomology over Dedekind rings}, Math. Z. 248 (2004), no. 4, 773--794. 
\bibitem{GL} T.~Geisser, M.~ Levine, {\em  The $K$-theory of fields in characteristic $p$},
 Invent. Math. {\bf 139} (2000), no. 3,
  459--493. 
\bibitem{GL2} T.~Geisser, M.~Levine, {\em The Bloch-Kato conjecture and a theorem of Suslin-
Voevodsky}, J. Reine Angew. Math. 530 (2001), 55--103.
\bibitem{Gr}  M.~Gros, {\em R\'egulateurs syntomiques et valeurs de fonctions L $p$-adiques. II.} Invent. Math. 115 (1994), no. 1, 61--79.  
\bibitem{Il} L.~Illusie, {\em Complexe de de Rham-Witt et cohomologie cristalline}, Annales Scientifiques
de l'\'E.N.S, 4e s\'erie, tome 12, no 4 (1979), 501--661.
\bibitem{K0} K.~Kato, {\em Galois cohomology of complete discrete valued fields},
 in Algebraic $K$
-Theory Part II, Oberwolfach 1980, 215--238, Lecture Notes in Mathematics 967 (1982), Springer Verlag.
\bibitem{K} K.~Kato, {\em On p-adic vanishing cycles (application of ideas
of Fontaine-Messing)}, Algebraic Geometry, Sendai, 1985, 
Adv. Stud. Pure Math. {\bf 10},
North-Holland, Amsterdam-New York, 1987, 207--251.
\bibitem{Kl} K.~Kato, {\em Logarithmic structures of Fontaine-Illusie},
 Algebraic analysis, geometry, and number theory (Baltimore,
  MD, 1988), 191--224, Johns Hopkins Univ. Press, Baltimore, MD, 1989.
\bibitem{K1} K.~Kato, {\em Semistable reduction and p-adic \'etale
cohomology}, Ast\'erisque {\bf 223}, Soc. Math. de France, 1994, 269--293.
\bibitem{KM} K.~Kato and  W.~Messing, {\em Syntomic cohomology and 
p-adic \'etale cohomology}, T\^ohoku Math. J. {\bf 44} (1992), 1--9.
\bibitem{Kur}  M.~Kurihara, {\em A note on $p$-adic \'etale cohomology}. Proc. Japan Acad. Ser. A Math. Sci. 63 (1987), no. 7, 275--278. 
\bibitem{Lor}  P.~Lorenzon, {\em Logarithmic Hodge-Witt forms and Hyodo-Kato cohomology}. J. Algebra 249 (2002), no. 2, 247--265. 
\bibitem{NN} J.~Nekov\'a\v{r}, W.~Nizio\l, {\em Syntomic cohomology and $p$-adic regulators for varieties over $p$-adic fields}, arXiv:1309.7620. 
\bibitem{Og} A.~Ogus, {\em The convergent topos in characteristic $p$}. The Grothendieck Festschrift, Vol. III, 133--162, Progr. Math., 88, Birkh\"auser Boston, Boston, MA, 1990. 
\bibitem{N2} W.~Nizio{\l},  {\em Crystalline Conjecture via K-theory}, 
 Ann. Sci. \'Ecole Norm. Sup. (4) {\bf 31} (1998), no. 5, 659--681.
 \bibitem{N1} W.~Nizio\l, {\em Cohomology of crystalline smooth sheaves.} Compositio Math. 129 (2001), no. 2, 123--147. 
\bibitem{ICM} W.~Nizio{\l}, {\em $p$-adic motivic cohomology in arithmetic}, International Congress of Mathematicians. Vol. II,  459--472, Eur. Math. Soc., Z\"urich, 2006.
\bibitem{N3} W.~Nizio{\l}, {\em Semistable Conjecture via K-theory}, Duke Math. J. 141 (2008), no. 1, 151-178.
\bibitem{New} W.~Nizio\l, {\em  On syntomic regulators, I: constructions}, preprint, 2016.
\bibitem{New1} W.~Nizio\l, {\em  On syntomic regulators, II: applications}, preprint, 2016.
\bibitem{Sh1} A.~Shiho, {\em Crystalline fundamental groups. II. Log convergent cohomology and rigid cohomology}. J. Math. Sci. Univ. Tokyo 9 (2002), no. 1, 1--163. 
\bibitem{SV} A.~Suslin, V.~Voevodsky, {\em Bloch-Kato conjecture and motivic cohomology with finite
coefficients}, The arithmetic and geometry of algebraic cycles (Banff, AB, 1998),
117-189, NATO Sci. Ser. C Math. Phys. Sci., 548, Kluwer Acad. Publ., Dordrecht,
2000.
\bibitem{Ts0} T.~Tsuji, {\em $p$-adic Hodge theory in the semi-stable reduction case}, Proceedings of the International Congress of Mathematicians, Vol. II (Berlin, 1998). Doc. Math. 1998, Extra Vol. II, 207--216 (electronic). \bibitem{Ts} T.~Tsuji, {\em p-adic \'etale and crystalline cohomology
in the semistable reduction case}, Invent. Math. {\bf 137} (1999), no. 2, 233--411. 
\bibitem{Ts1} T.~Tsuji, {\em On p-adic nearby cycles of log smooth families.} Bull. Soc. Math. France 128 (2000), no. 4, 529–575. 
\bibitem{Wei} C.~Weibel, {\em The norm residue isomorphism theorem}. Journal of Topology 2 (2) (2009), 346–372.
\end{thebibliography}
\end{document}